\documentclass{amsart}
\usepackage{amsmath,amssymb,amsthm,verbatim
}
\usepackage{tikz}
\theoremstyle{definition}
\newtheorem{theorem}{Theorem}
\newtheorem{lemma}{Lemma}
\newtheorem{corollary}{Corollary}
\newtheorem{proposition}{Proposition}

\begin{document}
\title{Generators for Tensor Product Components}
\author{Michael~J.~J.~Barry}
\address{Department of Mathematics\\
Allegheny College\\
Meadville, PA 16335}
\curraddr{15 River Street Unit 205\\
Boston, MA 02108}
\email{mbarry@allegheny.edu}

\subjclass[2010]{20C20}
\keywords{ cyclic group, indecomposable module, module generator, tensor product}

\begin{abstract}
Let $p$ be a prime number, $F$ a field of characteristic $p$, and $G$ a cyclic group of order $q =p^a$ for some positive integer $a$.   Under these circumstances every indecomposable $F G$-module is cyclic.  For indecomposable $F G$-modules $U$ and $W$, we present a new recursive method for identifying a generator for each of the indecomposable components of $U \otimes W$ in terms of a particular $F$-basis of $U \otimes W$.
\end{abstract}
\maketitle

\section{Introduction}
Let $p$ be a prime number, $F$ a field of characteristic $p$, and $G$ a cyclic group of order $q=p^a$.  Up to isomorphism, there is a unique indecomposable $F G$-module $V_q$ of dimension $q$~\cite[pp. 24--25]{Alperin}.  Let $g$ be a generator of $G$.  Then there is an ordered $F$-basis $(v_1,v_2,\dots,v_q)$ of $V_q$ such that $g v_1=v_1$ and $g v_i=v_{i-1}+v_i$ if $i>1$, that is, the matrix of $g$ with respect to this basis is a full Jordan block of eigenvalue $1$.  For an integer $i \in [1,q]$ define the vector space $V_i$ over $F$ by $V_i=\langle v_1,\dots,v_i \rangle$.  Then $V_i$ is an indecomposable $F G$-module and $\{V_1,\dots,V_q\}$ is a complete set of indecomposable $F G$-modules~\cite[pp. 24--25]{Alperin}.

Many authors have investigated the decomposition into indecomposables of $V_r \otimes V_s$ where $1 \leq r,s\leq q$ --- for example, in order of appearance see~\cite{Green1962,Bhama1964, Ralley1969,Renaud1979,McFaul1979,Norman1995,Hou2003,Norman2008,IandI2009}.  From the work of these authors, we know that
\[V_r \otimes V_s \cong V_{\lambda_1} \oplus V_{\lambda_2} \oplus \dots \oplus V_{\lambda_{\min(r,s)}}\]
where $\lambda_1 \geq \lambda_2 \geq \dots \geq \lambda_{\min(r,s)}>0$ but the dimensions $\lambda_\ell$ depend on the characteristic $p$.     Clearly $\{v_i \otimes v_j \mid 1 \leq i \leq r, 1 \leq j \leq s\}$ is an $F$-basis for $V_r \otimes V_s$.  However $\mathcal{B}=\{v_{i,j}=v_i \otimes g^{n-i}v_j \mid 1 \leq i \leq r,1 \leq j \leq s\}$ is also a basis for $V_r \otimes V_s$ that is nicer to work with because $(g-1)(v_{i,j})=v_{i-1,j}+v_{i.j-1}$~\cite[Lemma 4]{Norman1993}~or~\cite[Lemma 1]{B2017}.

For an integer $d$ satisfying $1 \leq d \leq r+s-1$,  define the vector subspace $D_d$ of $V_r \otimes V_s$ by
\[D_d=\langle v_{i,j} \mid i+j=d+1 \rangle.\]
Our preferred basis for $D_d$ is $\mathcal{B}_d=\{v_{i,j} \mid i+j=d+1\}$.
When $1 \leq \ell \leq \min(r,s)$,  $\mathcal{B}_{r+s-\ell}=(v_{r+1-\ell,s},v_{r+2-\ell,s-1},\dots,v_{r,s+1-\ell})$ and $\mathcal{B}_\ell=(v_{1,\ell},v_{2,\ell-1},\dots,v_{\ell,1})$.

When $1 \leq \ell \leq \min(r,s)$,  it is known, see~\cite[Lemma 4]{Norman2008} or~\cite[Theorem 2.2.2]{IandI2009}, that there is $y_\ell \in D_{r+s-\ell}$ such that $F G y_\ell\cong V_{\lambda_\ell}$ and
\[V_r \otimes V_s=F G y_1 \oplus F G y_2\oplus \dots \oplus F G y_{\min(r,s)}.\]
In this paper we will present a recursive algorithm for  identifying $y_\ell$ in terms of $\mathcal{B}_{r+s-\ell}$.  So $(g-1)^{\lambda_\ell}(y_\ell)=0$ but $(g-1)^{\lambda_\ell-1}(y_\ell)\neq 0$.  In addition, we want $(g-1)^{\lambda_\ell-1}(y_\ell)$ to be a specific element of $D_{r+s+1-\ell-\lambda_\ell}$.  For $1 \leq \ell \leq \min(r,s)$, define $x_\ell \in D_\ell$ by $x_\ell=\sum_{m=1}^\ell (-1)^{m-1}v_{m,\ell+1-m}$.  Note that $\{v \in D_\ell\mid (g-1)(v)=0\}=F x_\ell$~\cite[Lemma 3]{Norman2008}.    We will require $(g-1)^{\lambda_\ell-1}(y_\ell)=x_{r+s+1-\ell-\lambda_\ell}$.  

We now explain how recursion arises.  With $n$ as the unique nonnegative integer such that $p^n\leq \max(r,s)<p^{n+1}$ --- $\max(r,s)=p^n=p^a$ is a possibility, write $r=r_n p^n+R_{n-1}$ and $s=s_n p^n+S_{n-1}$ where $r_n$, $s_n$, $R_{n-1}$, and $S_{n-1}$ are integers satisfying $0 \leq r_n,s_n <p$ and $0 \leq R_{n-1},S_{n-1}<p^n$.     Let $m_{n-1}=\min(R_{n-1},S_{n-1})$ and $M_{n-1}=\max(R_{n-1},S_{n-1})$.    The base step for recursion is $m_{n-1}=0$ which includes the obvious case of $r=r_0p^0+0$ and $s=s_0 p^0+0$, but also when $n \geq 1$, the cases $(r,s)=(r_n p^n,s_n p^n)$, $(r,s)=(r_n p^n+R_{n-1},s_n p^n)$ with $R_{n-1}>0$, and $(r,s)=(r_n p^n,s_np^n+S_{n-1})$ with $S_{n-1}>0$.  Now suppose $\ell> \max(r+s-p^{n+1},0)$ and write $\ell=t p^n+k$ where $t$ and $k$ are integers with $t \geq 0$ and $1 \leq k \leq p^n$.  
It turns out that $\lambda_\ell=c p^n+C $ where $1 \leq C \leq p^n$ and $c$ is one of $(r_n+s_n-2t)$,  $(r_n+s_n-2t-1)$,  or $(r_n+s_n-2t-2)$. We will identify $z_\ell \in D_{r+s-\ell-c p^n}$ such that $(g-1)^{C-1}(z_\ell)=x_{r+s+1-\ell-\lambda_\ell}$.  This, which we believe is new,  is where most of the work will be.    To do this in the case when $m_{n-1}>0$, we will use results on $(R_{n-1},S_{n-1})$ and $(p^n-R_{n-1},p^n-S_{n-1})$ when $\max(1,R_{n-1}+S_{n-1}-p^n) \leq C \leq m_{n-1}$ and on $(R_{n-1},p^n-S_{n-1})$ and $(p^n-R_{n-1},S_{n-1})$ when $M_{n-1}+1 \leq C \leq \min(p^n,m_{n-1}+M_{n-1})$.   Recursion is not needed in order to identify $z_\ell$ for other values of $C$.  Then we will lift $z_\ell$ to $y_\ell \in D_{r+s-\ell}$ such that $(g-1)^{c p^n}(y_\ell)=z_\ell$ in a straightforward mechanical manner.  This lifting will involve the inverse of a matrix whose entries are binomials.  Of course the inverse of such a matrix can be written in terms of the determinant and the adjoint, but recent work by Nordenstam and Young~\cite{NordYoung}, specially tailored to this situation,  gives us a new formula for the $(i,j)$ entry of the inverse.

Previous work in this area includes a recursive algorithm by Norman~\cite{Norman2008} in a matrix setting and work by Iima and Iwamatsu~\cite{IandI2009} in the setting of a polynomial ring over two variables.  We give a much more explicit expression for a generator of a component than either of these works.  We see this work partly as a simplification and clarification of the ideas of Norman,  including identifying exactly where recursion is needed.    In addition, our approach is guided by a sharpened restatement of the results of Renaud~\cite{Renaud1979} and we incorporate recent work~\cite{NordYoung}  on inverses of matrices whose entries are binomials.  An earlier solution~\cite{BarryArX}, which initially failed to acknowledge how closely it mirrored ~\cite{IandI2009}, is quite different from the current work.  Some special cases along the lines of~\cite{BarryArX} were handled previously in~\cite{B2011} and~\cite{B2017}. 

We close this section by outlining the contents of the paper.   In Section~\ref{J} we define the $r \times s$ matrix $J(r,s)$, and we follow this in Section~\ref{SRen} by using the work of Renaud to identify $\lambda_\ell$.  Next we record some calculations in Section~\ref{Calc}, while in  Section~\ref{BinCoeffs} we apply the work of Nordenstam and Young.  Then in Section~\ref{FC} we deal with the case of $1 \leq \ell \leq \max(r+s-p^n,0)$.  In Section~\ref{TheCore} we show that some functions defined recursively take alternating values in $F$, which we then use in Sections~\ref{m>0} to identify $z_\ell$ and $y_\ell$ when $m_{n-1}>0$.  Finally Section~\ref{m=0} treats the case of $m_{n-1}=0$.

\section{$J(r,s)$}\label{J}

Loosely following Norman's notation~\cite{Norman2008}, let $J(r,s)$ be an $r \times s$ matrix over $F$ such that
when $1 \leq \ell \leq \min(r,s)$, 
\[(g-1)^{\lambda_\ell-1} \left(\sum_{j=1}^\ell J(r,s)_{\ell+1-j,j} v_{r-\ell+j,s+1-j}\right)=x_{r+s+1-\ell-\lambda_\ell},\]
that is, $\sum_{j=1}^\ell J(r,s)_{\ell+1-j,j} v_{r-\ell+j,s+1-j}$ is a candidate for $y_\ell$.
Thus identifying every $y_\ell$ amounts to specifying the first $\min(r,s)$ anti-diagonals of $J(r,s,p)$.  The other anti-diagonals will consist entirely of $0$s.   For certain $\ell$ there may be more than one candidate for $y_\ell$, so we must indicate how to choose $y_\ell$ in a systematic way.  The next two lemmas accomplish this.

Notation: With $1 \leq \ell \leq \min(r,s)$, define $\ell_0=\min\{i \mid \lambda_i=\lambda_\ell\}$ and $\ell_\infty=\max\{i \mid \lambda_i=\lambda_\ell\}$.

\begin{lemma}\label{Unique}
There is a unique element $y_{\ell_0} \in D_{r+s-\ell_0}$ such that $(g-1)^{\lambda_{\ell_0}-1}(y_{\ell_0})=x_{r+s+1-\ell_0-\lambda_{\ell_0}}$.
\end{lemma}
\begin{proof}
From~\cite[Lemma 4]{Norman2008} or~\cite[Theorem 2.2.2]{IandI2009} we know that $(g-1)^{\lambda_{\ell_0}-1}(y)=x_{r+s+1-\ell_0-\lambda_{\ell_0}}$ has a solution for $y \in D_{r+s-\ell_0}$.   It suffices to show that $(g-1)^{\lambda_{\ell_0}-1}$ is injective on $D_{r+s-\ell_0}$.
 By~\cite[Proposition 2]{B2015} or~\cite[Theorem 5]{GPX2},
\begin{align*}
\lambda_{\ell_0}&=r+s-2(\ell_0-1)-(\ell_\infty-(\ell_0-1))\\
&=r+s-\ell_0-\ell_\infty+1.
\end{align*}
Now $(g-1)^{\lambda_{\ell_0}-1}$, which maps $D_{r+s-\ell_0}$ to $D_{\ell_\infty}$, equals the $S_{\ell_\infty} \circ T$ where $T=(g-1)^{\ell_\infty-\ell_0}$ maps $D_{r+s-\ell_0}$ to $D_{r+s-\ell_\infty}$ and $S_{\ell_\infty}=(g-1)^{r+s-2 \ell_\infty}$ maps $D_{r+s-\ell_\infty}$ to $D_{\ell_\infty}$.  By the results in~\cite{IandI2009}  preceding and including Theorem~2.2.9, $S_{\ell_\infty}$ is invertible.  Since $T$ is injective by~\cite[Lemma 2.2.5]{IandI2009},  so is $(g-1)^{\lambda_{\ell_0}-1}$, and the result follows.
\end{proof}
The element $y_{\ell_0}$ corresponds to $\kappa_{(i)}$ in  Iima and Iwamatsu's notation for some $i$.
\begin{lemma}\label{Mult}
Assume that $\ell>\ell_0$ and that
\[y_{\ell_0}=\sum_{j=1}^{\ell_0} J(r,s)_{\ell_0+1-j,j}v_{r-\ell_0+j,s+1-j}.\]
Define $y_\ell \in D_{r+s-\ell}$ by
\[y_\ell=\sum_{j=1}^{\ell} J(r,s)_{\ell+1-j,j}v_{r-\ell+j,s+1-j}\]
where $J(r,s)_{\ell+1-j,j}=(-1)^{\ell-\ell_0}J(r,s)_{\ell_0+1-j,j}$ when $1 \leq j \leq \ell_0$ and $J(r,s)_{\ell+1-j,j}=0$ when $\ell_0+1 \leq j \leq \ell$.   Then $(g-1)^{\lambda_\ell-1}(y_\ell)=x_{r+s+1-\ell -\lambda_\ell}$.
\end{lemma}
\begin{proof}
Write $\lambda$ for $\lambda_{\ell_0}=\lambda_\ell$.
Since $(g-1)^{\lambda-1}(y_{\ell_0})=x_{r+s+1-\ell_0-\lambda}$, the coefficient of $v_{z, r+s+2-\ell_0-\lambda-z}$, where $1 \leq z \leq  r+s+1-\ell_0-\lambda$, in 
$(g-1)^{\lambda-1}(y_{\ell_0})$ is
\[\sum_{j=1}^{\ell_0} J(r,s)_{\ell_0+1-j,j}\binom{\lambda-1}{r-\ell_0+j-z}=(-1)^{z-1}.\]
On the other hand, the coefficient of $v_{z,r+s+2-\ell -\lambda-z}$, where $1 \leq z \leq  r+s+1-\ell -\lambda$, in $(g-1)^{\lambda-1}(y_\ell)$ is
\begin{align*}
\sum_{j=1}^{\ell} &J(r,s)_{\ell+1-j,j}\binom{\lambda-1}{r-\ell+j-z}\\
&=\sum_{j=1}^{\ell_0} (-1)^{\ell-\ell_0}J(r,s)_{\ell_0+1-j,j}\binom{\lambda-1}{r-\ell_0+j-(z+\ell-\ell_0)}\\
&=(-1)^{\ell-\ell_0}(-1)^{z+\ell-\ell_0-1}\\
&=(-1)^{z-1},
\end{align*}
proving $(g-1)^{\lambda-1}(y_\ell)=x_{r+s+1-\ell -\lambda}$.
\end{proof}
If $y_{\ell_0}$ corresponds $\kappa_{(i)}$ in  Iima and Iwamatsu's notation,  then $y_\ell$ corresponds to $\kappa_{(i )}x^{\ell-\ell_0}$.

\section{Identifying $\lambda_\ell$}\label{SRen}
Now we rewrite the work of Renaud~\cite{Renaud1979} to suit our needs.

\begin{theorem}\label{RenaudThm}
With $p$ a prime number and $r$ and $s$ positive integers, let
\[V_r \otimes V_s \cong V_{\lambda_1} \oplus V_{\lambda_2} \oplus \dots \oplus V_{\lambda_{\min(r,s)}}\]
where $\lambda_1 \geq \lambda_2 \geq \dots \geq \lambda_{\min(r,s)}>0$.

With $n$ as the unique nonnegative integer such that $p^n\leq \max(r,s)<p^{n+1}$,  write $r=r_n p^n+R_{n-1}$ and $s=s_n p^n+S_{n-1}$ where $r_n$, $s_n$, $R_{n-1}$, and $S_{n-1}$ are integers satisfying $0 \leq r_n,s_n <p$ and $0 \leq R_{n-1},S_{n-1}<p^n$. 

Let $m_{n-1}=\min(R_{n-1},S_{n-1})$ and $M_{n-1}=\max(R_{n-1},S_{n-1})$.   We define a $p^n$-tuple $W=(w_1,w_2,\dots,w_{p^n})$ of integers which depends on the values of $m_{n-1}$ and $M_{n-1}$.
If $m_{n-1}=0$,  $W=M_{n-1} \cdot (0)\oplus (p^n-M_{n-1}) \cdot (-p^n)$.  Now assume that $m_{n-1}>0$ and let
 \[V_{R_{n-1}}\otimes V_{S_{n-1}}
\cong V_{\mu_1} \oplus V_{\mu_2} \oplus \dots \oplus V_{\mu_{m_{n-1}}}\]
 where $\mu_1 \geq \mu_2 \geq \dots \geq \mu_{m_{n-1}}>0$.  If $m_{n-1}+M_{n-1} \leq p^n$, 
\begin{align*}
W&=(\mu_1,\mu_2,\dots,\mu_{m_{n-1}}) \oplus (M_{n-1}-m_{n-1}) \cdot (0)\\
&\quad \oplus (-\mu_{m_{n-1}},-\mu_{m_{n-1}-1},\dots,-\mu_1) \oplus (p^n-m_{n-1}-M_{n-1}) \cdot (-p^n).
\end{align*}
If $m_{n-1}+M_{n-1}>p^n$, 
\begin{align*}
W&=(\mu_1,\mu_2,\dots,\mu_{m_{n-1}}) \oplus (M_{n-1}-m_{n-1}) \cdot (0)\\
&\quad \oplus  (-\mu_{m_{n-1}},-\mu_{m_{n-1}-1},\dots,-\mu_{m_{n-1}+M_{n-1}-p^n+1}).
\end{align*}
If $r+s>p^{n+1}$ and $1 \leq \ell \leq r+s-p^{n+1}$, $\lambda_\ell=p^{n+1}$.  If $\max(r+s-p^{n+1},0)+1 \leq \ell \leq \min(r,s)$, write $\ell=t p^n+k$ where $t$ and $k$ are nonnegative integers and $1 \leq k \leq p^n$.  Then $\lambda_\ell=(r_n+s_n-2t)p^n+w_k$.
\end{theorem}

\begin{proof}
This follows from Renaud~\cite[Theorems ~1~and~2]{Renaud1979}.  We will illustrate with a discussion of the case of $r_n+s_n \geq p$ in our work which corresponds to the case of $r_0+s_0 \geq p$ in Norman.  Other cases can be handled in a similar fashion.  To apply Renaud's Theorem~2,  we need to ignore the multiplicities and assume that $V_{r_1} \otimes V_{s_1}=\sum_{j=1}^{m_{n-1}} V_{b_j}$ where $b_1 \geq b_2 \geq \dots \geq b_{m_{n-1}}>0$.  The term $(r+s-pq)V_{pq}$ in Renaud accounts for $\lambda_\ell=p^{n+1}$ when $1 \leq \ell \leq r+s-p^{n+1}$ and $r_n+s_n \geq p$ in our work.  
It is clear from Renaud's work that what one might call ``semi-periodicity'' $\lambda_{\ell+p^n}=\lambda_\ell-2p^n$ kicks in with the module of largest dimension after the initial $r+s-pq$ modules of dimension $pq$ and continues until $\ell=\min(r,s)-p^n$.  Thus it begins at $\ell=\max(r+s-p^{n+1},0)+1=r+s-p^{n+1}+1$ in our work.  For example when $r_0+s_0 \geq p$ and $r_1+s_1<q$ in Norman, the largest dimension of a module after the initial $r+s-pq$ modules of dimension $pq$ is $(s_0-r_0+2d_2-1)q=(s_0-r_0+2(p-s_0)-1)q=(2p-r_0-s_0-1)q$.  But in our work, since $r+s-p^{n+1}+1=(r_n+s_n-p)p^n+R_{n-1}+S_{n-1}+1$, $(t,k)=(r_n+s_n-p,R_{n-1}+S_{n-1}+1)$ and $\lambda_{r+s-p^{n+1}+1}=(r_n+s_n-2(r_n+s_n-p))p^n+w_{R_{n-1}+S_{n-1}+1}=(2p-r_n-s_n)p^n-p^n$.  Again in Norman when $r_0+s_0 \geq p$ and $r_1+s_1\geq q$, the largest dimension of a module after the initial $r+s-pq$ modules of dimension $pq$ is $(s_0-r_0+2(p-s_0-1))q+b_1=(2p-r_0-s_0-2)q+b_{r_1+s_1-q+1}$. because the $(q-r_1-s_1)  V_{(s_0-r_0+2d_2-1)q}=(q-r_1-s_1)  V_{(s_0-r_0+2(p-s_0)-1)q}$ and $\sum_{j=1}^{r_1+s_1-q}V_{(s_0-r_0+2d_1)q+b_j}=\sum_{j=1}^{r_1+s_1-q}V_{(s_0-r_0+2(p-s_0-1))q+b_j}$ cancel each other.  But in our work, since $r+s-p^{n+1}+1=(r_n+s_n-p)p^n+R_{n-1}+S_{n-1}+1$, $(t,k)=(r_n+s_n-p+1)p^n+R_{n-1}+S_{n-1}-p^n+1$, $\lambda_{r+s-p^{n+1}+1}=(r_n+s_n-2(r_n+s_n-p+1))p^n+w_{R_{n-1}+S_{n-1}-p^n+1}=(2p-r_n-s_n-2)p^n+\mu_{R_{n-1}+S_{n-1}-p^n+1}$.
\end{proof}
\begin{corollary}\label{RenaudCor}
Here we consider only the case of $m_{n-1}>0$, 
 \[V_{R_{n-1}}\otimes V_{S_{n-1}}
\cong V_{\mu_1} \oplus V_{\mu_2} \oplus \dots \oplus V_{\mu_{m_{n-1}}}\]
 where $\mu_1 \geq \mu_2 \geq \dots \geq \mu_{m_{n-1}}>0$, 
$\max(r+s-p^{n+1},0)+1 \leq \ell \leq \min(r,s)$, and $\ell=t p^n+k$ where $t$ and $k$ are nonnegative integers with $1 \leq k \leq p^n$.   Then with $e_n=m_{n-1}+M_{n-1}-p^n$,
\[\lambda_\ell=
\begin{cases}
(r_n+s_n-2t)p^n+p^n, & 1 \leq k \leq \max(e_n,0),\\
(r_n+s_n-2t)p^n+\mu_k, & \max(e_n,0)+1 \leq k \leq m_{n-1},\\
(r_n+s_n-2t)p^n,& m_{n-1}+1 \leq k \leq M_{n-1},\\
(r_n+s_n-2t)p^n-\mu_{m_{n-1}+M_{n-1}+1-k}, & M_{n-1}+1 \leq k \leq p^n-\max(-e_n,0),\\
(r_n+s_n-2t-1)p^n, &p^n-\max(-e_n,0)+1\leq k \leq p^n.
\end{cases}
\]
\end{corollary}
Note the first case is empty and the fourth case is $M_{n-1}+1 \leq k \leq M_{n-1}+m_{n-1}$ if $e_n \leq 0$, while the fifth case is empty if $e_n \geq 0$.

\section{How to Calculate}\label{Calc}
The following results are probably well-known but we assemble them here for convenience.   Recall that our calculations of binomial coefficients take place in a field $F$ of characteristic $p$.
\begin{lemma}
Let $n$ be a positive integer and let $v_{a,b} \in \mathcal{B}_{a+b-1}$.  
Then $(g-1)^n(v_{a,b})=\sum_{i=0}^n\binom{n}{n-i}v_{a-n+i,b-i}$ where $v_{a-n+i,b-i}=0$ if $a-n+i \leq 0$ or $b-i \leq 0$.
\end{lemma}
\begin{proof}
Induction on $n$ starting with $(g-1)(v_{a,b})=v_{a-1,b}+v_{a,b-1}$.
\end{proof}
\begin{lemma}
For a positive integer $n$, $\binom{p^n-1}{i}=(-1)^i$ for every integer $i \in [0,p^n-1]$..
\end{lemma}
\begin{proof}
First assume that $n=1$ and that $\binom{p-1}{i}=(-1)^i$.  We will show that $\binom{p-1}{i+1}=(-1)^{i+1}$.
\[\binom{p-1}{i+1}=\frac{(p-1)!}{(i+1)!(p-i-2)!}=\frac{(p-i-1)}{i+1}\binom{p-1}{i}=(-1)(-1)^i=(-1)^{i+1}.
\]
Now assume that $n \geq 2$.  The base $p$ expansion of $p^n-1$ is $(p-1)p^{n-1}+(p-1)p^{n-2}+\dots+(p-1)p+p-1$.  Suppose that the base $p$ expansion of $i$ is $i_{n-1} p^{n-1}+i_{n-2} p^{n-1}+\dots+i_1 p+i_0$.  Then in $F$ by Lucas's theorem~\cite{Fine}
\begin{align*}
\binom{p^n-1}{i}&=\binom{p-1}{i_{n-1}}\binom{p-1}{i_{n-2}}\dots \binom{p-1}{i_1}\binom{p-1}{i_0}\\
&=(-1)^{i_{n-1}+i_{n-2}+\dots+ i_1+i_0}.
\end{align*}
If $p$ is an odd prime,
\[(-1)^{i_{n-1}+i_{n-2}+\dots+ i_1+i_0}=(-1)^{i_{n-1}p^{n-1}+i_{n-2}p^{n-2}+\dots+ i_1 p+i_0}=(-1)^i.\]
On the other hand, $\binom{2^n-1}{i}=1=(-1)^i$.
\end{proof}
\begin{lemma}
For a positive integers $n$ and $c$ with $c<p$, $\binom{cp^n}{i}=0$ unless $i=j p^n$ with $0 \leq j \leq c$.
\end{lemma}
\begin{proof}
We may suppose  that $0 \leq i \leq c p^n$ and write the base $p$ expansion of $i$ as $i_n p^n+\dots+i_1 p+i_0$.  Then by Lucas's theorem~\cite{Fine}
\[\binom{cp^n}{i}=\binom{c}{i_n}\binom{0}{i_{n-1}}\dots \binom{0}{i_1}\binom{0}{i_0}.\]
Thus $\binom{cp^n}{i}=0$ unless $i_{n-1}=i_{n-2}=\dots=i_1=i_0=0$ and $i_n \leq c$.
\end{proof}
\begin{corollary}\label{p^n2}
For a positive integers $n$ and $c$ with $c<p$,
\[(g-1)^{c p^n}(v_{a,b})= \sum_{i=0}^c \binom{c p^n}{c p^n -i p^n} v_{a-cp^n+i p^n,b-ip^n}
=\sum_{i=0}^c \binom{c}{i} v_{a-cp^n+i p^n,b-ip^n}\]
where $v_{a-cp^n+i p^n,b-ip^n}=0$ if $a-cp^n+i p^n \leq 0$ or $b-ip^n \leq 0$.
\end{corollary}

\section{Matrices with binomial entries}\label{BinCoeffs}
We will use invertible matrices with binomial entries in lifting $z_\ell$ to $y_\ell$.

\begin{proposition}\label{Adjoint}
With $a$ and $b$ nonnegative integers, let $M$ be the $k \times k$ matrix with $(i,j)$-entry $\binom{a}{b+j-i} \in F$ and suppose that $M$ is invertible.  Let $d_k=\prod_{i=0}^{k-1} \binom{a+i}{b}/\binom{b+i}{b} \in \mathbb{Q}$, and for each $(i,j)$ with $1 \leq i,j \leq k$, define
\[
z_{i,j}=\frac{d_k}{\binom{a+k-1}{b+i-1}}\sum_{\ell=1}^j(-1)^{\ell+j}\binom{a+k-1}{\ell-1}\binom{a+j-\ell-1}{j-\ell}\prod_{r=1,r \neq i}^k \frac{\ell-b-r}{i-r}
\in \mathbb{Q}.
\]
Then $d_k \in \mathbb{Z}$ and $z_{i,j} \in \mathbb{Z}$ for every $(i,j)$.  Let $\phi$ is the natural ring homomorphism of $\mathbb{Z}$ into $F$.   Since $M$ is invertible, $\det M=\phi(d_k) \neq 0 \in F$.   Define the $k \times k$ matrix $N$ whose $(i,j)$-entry is $\phi(d_k)^{-1}\phi(z_{i,j})$.  Then  $M^{-1}=N$.
\end{proposition}
\begin{proof}
Let $M_\mathbb{Q}$ be the $k \times k$ matrix with $(i,j)$-entry $\binom{a}{b+j-i} \in \mathbb{Q}$.  Then $\det M_\mathbb{Q} \in \mathbb{Z}$, actually equal to $d_k$ by a result of Roberts~\cite{Roberts}, and $adj(M_\mathbb{Q})$, the adjoint of $M_\mathbb{Q}$, is a matrix of integers, satisfying
$M_\mathbb{Q} adj(M_\mathbb{Q})= d_k I_k=adj(M_\mathbb{Q}) M_\mathbb{Q}$.
Replacing  $(A,B_j,n)$ by $(a,b+j,k)$ in~\cite[Theorem 1]{NordYoung}, we see that $z_{i,j}$ is the $(i,j)$-entry of $adj(M_\mathbb{Q})$.  Now apply $\phi$ to conclude that $M adj(M)=\phi(d_k) I_k=(\det M)I_k=adj(M) M$ and $M^{-1}=(\det M)^{-1} adj(M)=N$.  
\end{proof}

\begin{lemma}\label{pdiv}
Let $a$ and $b$ be nonnegative integers and let $M$ be the $k \times k$ matrix with $(i,j)$ entry $\binom{a}{b+j-i}\in F$.   If $b \leq a$ and $a+k-1<p$, then $M$ is invertible.
\end{lemma}
\begin{proof}
Let $M_{\mathbb{Z}}$ be the $k \times k$ matrix with $(i,j)$ entry $\binom{a}{b+j-i}\in \mathbb{Z}$.  Then by~\cite{Roberts}, $\det M_{\mathbb{Z}}
=\prod_{z=0}^{k-1} \binom{a+z}{b}/\binom{b+z}{b}$.  The result follows.
\end{proof}

\section{First Case}\label{FC}

\begin{lemma}\label{Overrun}
Suppose that $r+s>p^{n+1}$ and $1 \leq \ell \leq r+s-p^{n+1}$,  so $\lambda_\ell=p^{n+1}$ by Theorem~\ref{RenaudThm}.  Let $y_\ell=(-1)^{r-\ell}v_{r+1-\ell,s}$.  Then $(g-1)^{p^{n+1}-1}(y_\ell)=x_{r+s+1-\ell-p^{n+1}}$. 
\end{lemma}
\begin{proof}
Note our choice of $y_\ell$ for $\ell>1$ is in line with the procedure discussed in Section~\ref{J}.  Now $(g-1)^{p^{n+1}-1}$ maps $D_{r+s-\ell}$ to $D_{r+s+1-\ell-p^{n+1}}$ which has basis 
\[\{v_{z,r+s+2-\ell-p^{n+1}-z} \mid 1 \leq z \leq r+s+1-\ell-p^{n+1}\}.\]
The coefficient of $v_{z,r+s+2-\ell-p^{n+1}-z}$ in $x_{r+s+1-\ell-p^{n+1}}$ is
\[(-1)^{r-\ell} \binom{p^{n+1}-1}{r+1-\ell-z}=(-1)^{r-\ell}(-1)^{r+1-\ell-z}=(-1)^{z-1},
\]
proving that $(g-1)^{p^{n+1}-1}(y_\ell)=x_{r+s+1-\ell-p^{n+1}}$. 
\end{proof}
For the rest of this paper we assume that $\max(r+s-p^{n+1},0)+1 \leq \ell \leq \min(r,s)$.

\section{Alternating $F$-valued Functions}\label{TheCore}
In this section $m_{n-1}>0$ and
\[V_{R_{n-1}}\otimes V_{S_{n-1}}
\cong V_{\mu_1} \oplus V_{\mu_2} \oplus \dots \oplus V_{\mu_{m_{n-1}}}\]
where $\mu_1 \geq \mu_2 \geq \dots \geq \mu_{m_{n-1}}>0$.  
\begin{lemma}\label{Tensors}
Let $e_n=m_{n-1}+M_{n-1}-p^n$ and $f_n=S_{n-1}-R_{n-1}$.  Then 
 \begin{enumerate}
 \item $V_{R_{n-1}}\otimes V_{S_{n-1}}
\cong \max(e_n,0) \cdot V_{p^n} 
\oplus V_{\mu_{\max(e_n,0)+1}} \oplus \dots \oplus V_{\mu_{m_{n-1}}}$ \\with $\mu_{\max(e_n,0)+1}<p^n$,
 \item $V_{p^n-R_{n-1}} \otimes V_{p^n-S_{n-1}}\cong
\max(-e_n,0) \cdot V_{p^n}
 \oplus V_{\mu_{\max(e_n,0)+1}} \oplus \dots \oplus V_{\mu_{m_{n-1}}}$,
  \item $
 V_{p^n-R_{n-1}} \otimes V_{S_{n-1}}\cong
 \max(f_n,0)\cdot V_{p^n}
\oplus V_{p^n-\mu_{m_{n-1}}} \oplus \dots \oplus V_{p^n-\mu_{\max(e_n,0)+1}}$,

and

 \item $V_{R_{n-1}} \otimes V_{p^n-S_{n-1}}\cong
 \max(-f_n,0)\cdot V_{p^n}
\oplus V_{p^n-\mu_{m_{n-1}}} \oplus \dots \oplus V_{p^n-\mu_{\max(e_n,0)+1}}$.
\end{enumerate}
\end{lemma}
\begin{proof}
From~\cite[Corollary 1 part (1)]{B2011_0} it follows that if $e_n >0$, then
\[V_{R_{n-1}} \otimes V_{S_{n-1}} \cong (e_n) \cdot V_{p^n} \oplus 
V_{p^n-R_{n-1}} \otimes V_{p^n-S_{n-1}}\]
and $\mu_{e_n+1}\leq p^n-R_{n-1}+p^n-S_{n-1}-1 <p^n$,
whereas if $e_n \leq 0$,
\[V_{p^n-R_{n-1}} \otimes V_{p^n-S_{n-1}}\cong (-e_n) \cdot V_{p^n} \oplus 
V_{R_{n-1}} \otimes V_{S_{n-1}}\]
and $\mu_1 \leq R_{n-1}+S_{n-1}-1<p^n$.  We have proved (1) and (2).

The result for $
 V_{p^n-R_{n-1}} \otimes V_{S_{n-1}}$ follows from~\cite[(2.5a)]{Green1962} or more carefully from~\cite[Proposition 1]{GPX2} proving (3).  Finally (4) follows from (1)
because $V_{R_{n-1}} \otimes V_{p^n-S_{n-1}}$ is related to $V_{p^n-R_{n-1}} \otimes V_{S_{n-1}}$ as $ V_{p^n-R_{n-1}} \otimes V_{p^n-S_{n-1}}$ is to $V_{R_{n-1}} \otimes V_{S_{n-1}}$.
\end{proof}

\textbf{Notation:} With $1 \leq k \leq m_{n-1}$, let $k_0=\min\{i \mid \mu_i=\mu_k\}$,  $k_\infty=\max\{i \mid \mu_i=\mu_k\}$, and $k'=k_0+k_\infty-k$.  Note that $k_0 \leq k' \leq k_\infty$.  

\begin{lemma}\label{Multiplicity}
Let  $1 \leq k \leq m_{n-1}$.  
Then $\mu_k=R_{n-1}+S_{n-1}-k_0-k_\infty+1$.
\end{lemma}
\begin{proof}
By~\cite[Proposition 2]{B2015} or~\cite[Theorem 5]{GPX2},
\begin{align*}
\mu_k&=R_{n-1}+S_{n-1}-2(k_0-1)-(k_\infty-(k_0-1))\\
&=R_{n-1}+S_{n-1}-k_0-k_\infty+1.
\end{align*}
\end{proof}

\begin{lemma}\label{SsubN-1}
$\max(f_n,0)+m_{n-1}=S_{n-1}$.
\end{lemma}
\begin{proof}
If $R_{n-1} \leq S_{n-1}$, $\max(f_n,0)+m_{n-1}=(S_{n-1}-R_{n-1})+R_{n-1}=S_{n-1}$,
while if $R_{n-1} >S_{n-1}$, $\max(f_n,0)+m_{n-1}=0+S_{n-1}=S_{n-1}$.
\end{proof}

 \begin{lemma}\label{nuTomu}
When $\max(e_n,0)+1 \leq k \leq m_{n-1}$ and
\[V_{p^n-R_{n-1}} \otimes V_{S_{n-1}}  
\cong \bigoplus_{i=1}^{\min(p^n-R_{n-1},S_{n-1})} V_{\nu_i}\]
where $\nu_1 \geq \nu_2 \geq \dots \geq \nu_{\min(p^n-R_{n-1},S_{n-1})}>0$,
then $\nu_{S_{n-1}+1-k}=p^n-\mu_k$. 
 \end{lemma}
\begin{proof}
By Lemma~\ref{Tensors}, $p^n-\mu_k=\nu_{\max(f_n,0)+m_{n-1}+1-k}$.   Since
$\max(f_n,0)+m_{n-1}=S_{n-1}$ by Lemma~\ref{SsubN-1}, the result follows.
\end{proof}

\begin{theorem}\label{ZThm}
Assume that $\max(e_n,0)+1 \leq k \leq m_{n-1}$ and  that $A=J(R_{n-1},S_{n-1})$ and $D=J(p^n-R_{n-1},p^n-S_{n-1})$ are known.  For a nonnegative integer $i$ and a positive integer $z$ define
\[A_i(z)=\sum_{j=1}^k a_{k+1-j,j}\binom{\mu_k-1}{i p^n+R_{n-1}-k+j-z},\]
and
\[D_i(z)=\sum_{j=1}^{k-e_n}d_{k-e_n+1-j,j}\binom{\mu_k-1}{i p^n+R_{n-1}+S_{n-1}+j-k-z}.\]
For a nonnegative integer $t$ define 
\[f_t(z)=\sum_{i=0}^t (-1)^i A_i(z) +(-1)^{R_{n-1}}\sum_{i=0}^{t}(-1)^i D_i(z).\]
Then $f_t(z)=(-1)^{z-1}$ when $1 \leq z \leq t p^n+p^n-S_{n-1}+k'$.
\end{theorem}
Clearly $A_i$, $D_i$, and $f_t$ depend on $p^n$, $R_{n-1}$, $S_{n-1}$, and $k$, but we want to keep the notation simple.   We will prove Theorem~\ref{ZThm} in a series of lemmas.

\begin{lemma}\label{Reduction}
When $i p^n+1 \leq z \leq (i+1) p^n$, only $A_i(z)$, $D_i(z)$, and $A_{i+1}(z)$ contribute to $f_t(z)$.  Specifically we will show that
\[f_t(z)=
\begin{cases}
(-1)^i A_i(z), & \text{$i p^n +1 \leq z \leq i p^n+k'$,}\\
(-1)^i\left(A_i(z)+(-1)^{R_{n-1}}D_i(z)\right), &\text{$i p^n+ k'+1 \leq z \leq i p^n+R_{n-1}$,}\\
(-1)^i(-1)^{R_{n-1}}D_i(z), & \text{$i p^n +R_{n-1}+1 \leq z \leq ip^n+a_1$,}\\
(-1)^i\left((-1)^{R_{n-1}}D_i(z)-A_{i+1}(z)\right), &\text{$i p^n+a_1+1 \leq z\leq (i+1)p^n$}
\end{cases}
\]
where  $a_1=p^n-S_{n-1}+k'$.
\end{lemma}

\begin{proof}
First we show that $A_{i+1}(z)=0$ if $z \leq  i p^n+p^n-S_{n-1}+k'$, $D_i(z)=0$ if $z \leq i p^n+k' $ or $z >(i+1) p^n$, and $A_i(z)=0$ if $z>i p^n+R_{n-1}$.

Suppose that $z \leq  i p^n+p^n-S_{n-1}+k'$.  Then 
\begin{align*}
(i+1)&p^n+R_{n-1}-k+j-z\\
& \geq (i+1)p^n+R_{n-1}-k+j-(i p^n+p^n-S_{n-1}+k')\\
&=R_{n-1}+S_{n-1}-k_0-k_\infty +j\\
&=\mu_k+j-1\\
&>\mu_k-1.
\end{align*}
Hence $A_{i+1}(z)=0$.  Clearly if $\ell>i+1$, $A_{\ell}(z)=0$.

Suppose that $z \leq i p^n+k' $.  Then
\begin{align*}
i p^n+R_{n-1}+S_{n-1}+j-k-z
&\geq i p^n+R_{n-1}+S_{n-1}+j-k-(i p^n+ k')\\
&=\mu_k+j-1\\
&>\mu_k-1.
\end{align*}
Hence $D_i(z)=0$.  Clearly if $\ell>i$,  $D_{\ell}(z)=0$.

Suppose that $z>(i+1)p^n$.  Then since $j \leq k+p^n-R_{n-1}-S_{n-1}$
\begin{align*}
i p^n&+R_{n-1}+S_{n-1}+j-k-z\\
&\leq i p^n+R_{n-1}+S_{n-1}+(p^n+k-R_{n-1}-S_{n-1})-k-z\\
&=(i+1)p^n-z\\
&<0.
\end{align*}
Hence $D_i(z)=0$.  Clearly if $\ell<i$, $D_{\ell}(z)=0$.

Suppose that $z>ip^n+R_{n-1}$.  Then since $j \leq k$
\begin{align*}
i p^n&+R_{n-1}+j-k-z\\
&=(i p^n+R_{n-1}-z)+(j-k)\\
&\leq i p^n+R_{n-1}-z\\
&<0.
\end{align*}
Hence $A_i(z)=0$.  Clearly if $\ell<i$, $A_\ell(z)=0$.

The detailed behavior of $f_t$ when $i p^n+1 \leq z \leq (i+1) p^n$ follows.
\end{proof}

\begin{corollary} \label{period}
When $1 \leq z \leq t p^n-S_{n-1}+k'$, $f_t(z+p^n)=-f_t(z)$.
\end{corollary}
\begin{proof}
Follows because $A_{i+1}(z+p^n)=A_i(z)$ and $D_{i+1}(z+p^n)=D_i(z)$.  For example
\begin{align*}
A_{i+1}(z+p^n)&=\sum_{j=1}^k a_{k+1-j,j}\binom{\mu_k-1}{(i+1) p^n+R_{n-1}-k+j-(z+p^n)}\\
&=\sum_{j=1}^k a_{k+1-j,j}\binom{\mu_k-1}{i p^n+R_{n-1}-k+j-z}\\
&=A_i(z).
\end{align*}
\end{proof}

 Our next two lemmas will conclude the  proof of Theorem~\ref{ZThm}.

\begin{lemma}\label{AiDi}
When $ip^n+1 \leq z \leq (i+1)p^n+R_{n-1}$, 
$(-1)^i (A_i(z)+(-1)^{R_{n-1}}D_i(z))=(-1)^{z-1}$.  Since $D_i(z)=0$ when $z \leq i p^n+k'$ by Lemma~\ref{Reduction}, this implies that $(-1)^i A_i(z)=(-1)^{z-1}$ when $i p^n+1 \leq z \leq i p^n+k'$.
\end{lemma}
\begin{proof}
Because of Corollary~\ref{period},  it suffices to prove 
\begin{equation}\label{Eq1}
A_0(z)+(-1)^{R_{n-1}}D_0(z)=(-1)^{z-1} \qquad \text{ when }1 \leq z \leq R_{n-1}
\end{equation}
since if $f(z)=(-1)^{z-1}$, then $f(z+i p^n)=(-1)^i f(z)=(-1)^{ip^n}(-1)^{z-1}=(-1)^{z+ip^n-1}$.

Step 1: $V_{p^n-R_{n-1}} \otimes V_{S_{n-1}}$.  By Lemma~\ref{Tensors},
\[V_{p^n-R_{n-1}} \otimes V_{S_{n-1}}\cong
\max(f_n,0)\cdot V_{p^n}
\oplus V_{p^n-\mu_{m_{n-1}}} \oplus \dots \oplus V_{p^n-\mu_{\max(e_n,0)+1}}.\]

Here we adapt an argument of Norman~\cite[Theorem 3]{Norman2008}.  By ~\cite[Lemma 4]{Norman2008} or~\cite[Theorem 2.2.2]{IandI2009}, there exits $w_\ell \in D_{p^n-R_{n-1}+S_{n-1}-\ell}$ such that $F G w_\ell \cong V_{p^n}$ if $1 \leq \ell \leq \max(f_n,0)$, $F G w_\ell \cong V_{p^n-\mu_{m_{n-1}+1-(\ell- \max(f_n,0))}}$ if $ \max(f_n,0)+1 \leq \ell \leq  \min(p^n-R_{n-1},S_{n-1})$, and
\[V_{p^n-R_{n-1}} \otimes V_{S_{n-1}} = \bigoplus_{\ell=1}^{\min(p^n-R_{n-1},S_{n-1}) }F G w_\ell.\]

Then $((g-1)^{\ell-1}(w_1),(g-1)^{\ell-2}(w_2),\dots,(g-1)(w_{\ell-1}),w_\ell)$ is a basis for $D_{p^n-R_{n-1}+S_{n-1}-\ell}$.  In particular, $w_\ell \notin (g-1)(D_{p^n-R_{n-1}+S_{n-1}-\ell+1})$.

Let
\[w_\ell=\sum_{j=1}^\ell \beta_{\ell+1-j,j}v_{p^n-R_{n-1}-\ell+j,S_{n-1}+1-j} \in D_{p^n-R_{n-1}+S_{n-1}-\ell}.\]

By Lemma~\ref{nuTomu}, when $\max(e_n,0)+1 \leq k \leq m_{n-1}$,  the $(S_{n-1}+1-k)$th summand of $V_{p^n-R_{n-1}} \otimes V_{S_{n-1}} $ in order of non-increasing dimension is $V_{p^n-\mu_k}$.

Now
\begin{align*}
w_{S_{n-1}+1-k'}
&=\sum_{j=1}^{S_{n-1}+1-k'} \beta_{S_{n-1}+2-k'-j,j}v_{p^n-R_{n-1}-(S_{n-1}+1-k')+j,S_{n-1}+1-j} \\
&=\sum_{j=1}^{S_{n-1}+1-k'} \beta_{S_{n-1}+2-k'-j,j}v_{p^n+k'-R_{n-1}-S_{n-1}+j-1,S_{n-1}+1-j} \\
& \quad \in D_{p^n+k'-R_{n-1}-1} \subset V_{p^n-R_{n-1}}\otimes V_{S_{n-1}}.
\end{align*}
Since $(g-1)^{p^n-\mu}(w_{S_{n-1}+1-k'})=0 \in D_{k'+\mu-R_{n-1}-1}=D_{S_{n-1}-k}$, when $1 \leq z \leq S_{n-1}-k$, the coefficient of $v_{z,S_{n-1}-k+1-z}$ in $(g-1)^{p^n-\mu}(w_{S_{n-1}+1-k'})$ is
\[\sum_{j=1}^{S_{n-1}+1-k'} \beta_{S_{n-1}+2-k'-j,j}\binom{p^n-\mu}{p^n+k'-R_{n-1}-S_{n-1}+j-1-z}=0.\]

Step 2: $V_{p^n} \otimes V_{p^n}$.  By~\cite[Lemma 4]{Norman2008} or~\cite[Theorem 2.2.2]{IandI2009},when $1 \leq \ell \leq p^n$, there exists $u_\ell \in D_{2 p^n-\ell}$ such that $F G u_\ell \cong V_{p^n}$ and
\[V_{p^n} \otimes V_{p^n} = \bigoplus_{\ell=1}^{p^n} F G u_\ell.\]
Hence $(g-1)^{p^n}(u)=0$ for every $u \in V_{p^n} \otimes V_{p^n}$.

When $1 \leq \ell \leq \min(p^n-R_{n-1},S_{n-1})$, let
\[y_\ell=\sum_{j=1}^\ell \beta_{\ell+1-j,j} v_{p^n-\ell+j,p^n+1-j} \in D_{2p^n-\ell}.
\]
We claim that $(g-1)^{p^n-1}(y_\ell) \neq 0$.  
\begin{proof}[Proof of Claim]
First $((g-1)^{\ell-1}(u_1),(g-1)^{\ell-2}(u_2),\dots,(g-1)(u_{\ell-1}),u_\ell)$ is a basis for $D_{2 p^n-\ell}$.  Write
\[y_\ell=\epsilon_1 (g-1)^{\ell-1}(u_1)+\epsilon_2 (g-1)^{\ell-2}(u_2)+\dots+\epsilon_{\ell-1} (g-1)(u_{\ell-1})+
\epsilon_\ell u_\ell.\]
We argue by contradiction.  If $(g-1)^{p^n-1}(y_\ell) = 0$,  then
\[0=(g-1)^{p^n-1}(y_\ell) =0+\epsilon_\ell(g-1)^{p^n-1}(u_\ell)\]
forcing $\epsilon_\ell=0$ and $y_\ell=(g-1)(u)$ where
\[u=\epsilon_1 (g-1)^{\ell-2}(u_1)+\epsilon_2 (g-1)^{\ell-3}(u_2)+\dots+\epsilon_{\ell-1} u_{\ell-1}
\in D_{2p^n+1-\ell}.\]
But if $y_\ell \in (g-1)(D_{2p^n+1-\ell})$, then $w_\ell \in (g-1)(D_{p^n-R_{n-1}+S_{n-1}-\ell+1})$ --- a contradiction!
\end{proof}
Then $(g-1)^{p^n-1}(y_\ell)=b_\ell x_{p^n+1-\ell}$ where $b_\ell \in F \setminus \{0\}$ since $(g-1)^{p^n-1}(y_\ell)\in \{v \in D_ {p^n+1-\ell}\mid g(v)=v\}=\{\alpha x_{p^n+1-\ell}\mid \alpha \in F\}$~\cite[Lemma 3]{Norman2008}.  Scale $y_\ell$ so that $(g-1)^{p^n-1}(y_\ell)=x_{p^n+1-\ell}$.

In $V_{p^n} \otimes V_{p^n}$,
\[
y_{S_{n-1}+1-k'}
=\sum_{j=1}^{S_{n-1}+1-k'} \beta_{S_{n-1}+2-k'-j,j}v_{p^n+k'-S_{n-1}-1+j,p^n+1-j}
\in D_{2p^n+k'-S_{n-1}-1}.
\]

Then $(g-1)^{p^n -\mu}(y_{S_{n-1}+1-k'})\in D_{p^n+k'+\mu-S_{n-1}-1}=D_{p^n+R_{n-1}-k}$.  The usual basis for $ D_{p^n+R_{n-1}-k}$ is
\[\mathcal{B}_{p^n+R_{n-1}-k}=(v_{R_{n-1}-k+1,p^n},\dots, v_{p^n,R_{n-1}-k+1}),\]
so $\dim D_{p^n+R_{n-1}-k}=p^n+k-R_{n-1}$.

Thus
when $1 \leq z \leq S_{n-1}-k$, the coefficient of $v_{R_{n-1}+z, p^n-k+1-z}$ in $(g-1)^{p^n -\mu}(y_{S_{n-1}+1-k'})$ is
\[\sum_{j=1}^{S_{n-1}+1-k'} \beta_{S_{n-1}+2-k'-j,j}\binom{p^n-\mu}{p^n+k'-S_{n-1}-1+j-R_{n-1}-z},\]
which equals $0$ from above.

We have just seen that 
the elements  
\[v_{R_{n-1}+1,p^n-k},v_{R_{n-1}+2,p^n-k-1},\dots,v_{R_{n-1}+S_{n-1}-k,p^n+1-S_{n-1}}
\]
of $\mathcal{B}_{p^n+R_{n-1}-k}$ occur with coefficient $0$ in $(g-1)^{p^n-\mu}(y_{S_{n-1}+1-k'}) $. The other elements of $\mathcal{B}_{p^n+R_{n-1}-k}$ come in two groups: $v_{R_{n-1}-k+1,p^n},\dots,v_{R_{n-1},p^n-k+1}$ with $k$ vectors and
$v_{R_{n-1}+S_{n-1}+1-k,p^n-S_{n-1}}, \dots, v_{p^n,R_{n-1}-k+1}$ with $p^n+k-R_{n-1}-S_{n-1}$ vectors.

Write $(g-1)^{p^n-\mu}(y_{S_{n-1}+1-k'}) \in D_{p^n+R_{n-1}-k}$ as
\[z_k=\sum_{j=1}^k \alpha'_j v_{R_{n-1}-k+j,p^n+1-j}+\sum_{j=1}^{k-e_n}\delta'_j v_{R_{n-1}+S_{n-1}+j-k,p^n-S_{n-1}+1-j}.\]
So $(g-1)^{\mu-1}(z_k)=x_{p^n+R_{n-1}+1-k-\mu}=x_{p^n-S_{n-1}+k'}$.

Step 3: Identifying $\alpha'_j$ and $\delta'_j$ when $k=k_0$.  When $z < k'+1$, the coefficient of $v_{z,p^n-S_{n-1}+k+1-z}$ in 
\[(g-1)^{\mu-1}(\sum_{j=1}^{k-e_n}\delta'_j v_{R_{n-1}+S_{n-1}+j-k,p^n-S_{n-1}+1-j})\] 
is $0$ because $R_{n-1}+S_{n-1}+j-k-(\mu-1)=j+k'>k'$ by Lemma~\ref{Multiplicity}.

Thus when $1 \leq z \leq k'$, the coefficient of $v_{z,p^n-S_{n-1}+k+1-z}$ in $(g-1)^{\mu-1}(z_k)$ is
\[\sum_{j=1}^k \alpha'_j \binom{\mu-1}{R_{n-1}-k+j-z}=(-1)^{z-1}.\]

In $V_{R_{n-1}} \times V_{S_{N-1}}$, when $\max(e_n,0)+1 \leq k \leq m_{n-1}$, 
\[(g-1)^{\mu-1}\left(\sum_{j=1}^k a_{k+1-j,j}v_{R_{n-1}-k+j,S_{n-1}+1-j}\right)=x_{k_0+k_\infty-k}.\]

Hence the coefficient of $v_{z,k'+1-z}$ in $(g-1)^{\mu-1}\left(\sum_{j=1}^k a_{k+1-j,j}v_{R_{n-1}-k+j,S_{n-1}+1-j}\right)$ is
\[
\sum_{j=1}^k a_{k+1-j,j}\binom{\mu-1}{R_{n-1}-k+j-z}=(-1)^{z-1}.\]
If $k=k_0$, there is a unique $y \in D_{R_{n-1}+S_{n-1}-k_0}$ such that $(g-1)^{\mu-1}(y)=x_{k_0+k_\infty-k}$ by Lemma~\ref{Unique}.  It follows that $\alpha'_j= a_{k+1-j,j}$ when $1 \leq j \leq k$.

When $z>R_{n-1}$, the coefficient of $v_{z,p^n-S_{n-1}+1+k'-z}$ in \\$(g-1)^{\mu-1}(\sum_{j=1}^k \alpha'_j v_{R_{n-1}-k+j,p^n+1-j})$ is $0$ because $R_{n-1}-k+j \leq R_{n-1}$.

Thus when $1 \leq z \leq p^n-R_{n-1}-S_{n-1}+k'$,  the coefficient of $v_{R_{n-1}+z,p^n-R_{n-1}-S_{n-1}+k'+1-(R_{n-1}+z)}$ in $(g-1)^{\mu-1}(z_k)$ is 
\[\sum_{j=1}^{k-e_n}\delta'_j \binom{\mu-1}{R_{n-1}+S_{n-1}+j-k-(R_{n-1}+z)}=(-1)^{R_{n-1}+z-1},\]
or
\[\sum_{j=1}^{k-e_n}\delta'_j \binom{\mu-1}{S_{n-1}+j-k-z}=(-1)^{R_{n-1}+z-1}.\]

\
In $V_{p^n-R_{n-1}} \otimes V_{p^n-S_{n-1}}$, when $\max(e_n,0)+1 \leq k \leq m_{n-1}$, 
\[(g-1)^{\mu-1}\left(\sum_{j=1}^{k-e_n} d_{k+p^n+1-R_{n-1}-S_{n-1}-j,j}v_{S_{n-1}-k+j,p^n-S_{n-1}+1-j}\right)=
x_{p^n-R_{n-1}-S_{n-1}+k'}.\]

Hence the coefficient of $v_{z,p^n-R_{n-1}-S_{n-1}+k'+1-z}$ in
\[(g-1)^{\mu-1}\left(\sum_{j=1}^{k-e_n} d_{k+p^n+1-R_{n-1}-S_{n-1}-j,j}v_{S_{n-1}-k+j,p^n-S_{n-1}+1-j}\right)\] is
\[\sum_{j=1}^{k-e_n} d_{k+p^n+1-R_{n-1}-S_{n-1}-j,j}\binom{\mu-1}{S_{n-1}-k+j-z}=(-1)^{z-1}.\]

If $k=k_0$,  there is a unique $y \in D_{p^n-R_{n-1}+p^n-S_{n-1}-k_0}$ such that $(g-1)^{\mu-1}(y)=x_{p^n-R_{n-1}-S_{n-1}+k'}$ by Lemma~\ref{Unique}.  It follows that $\delta'_j =(-1)^{R_{n-1}} d_{k+p^n+1-R_{n-1}-S_{n-1}-j,j}$ when $1 \leq j \leq k-e_n$.   Hence
\begin{align*}
z_{k_0}&=\sum_{j=1}^{k_0} a_{k_0+1-j,j} v_{R_{n-1}-k_0+j,p^n+1-j}\\
&\quad+(-1)^{R_{n-1}}\sum_{j=1}^{k_0-e_n}d_{k_0-e_n+1-j,j}v_{R_{n-1}+S_{n-1}+j-k_0,p^n-S_{n-1}+1-j}.
\end{align*}
Since $(g-1)^{\mu-1}(z_{k_0})=x_{p^n-S_{n-1}+k_\infty}$ the coefficient of $v_{z,p^n-S_{n-1}+k_\infty+1-z}$ in $(g-1)^{\mu-1}(z_{k_0})$ is
\begin{align*}
(-1)^{z-1}&=\sum_{j=1}^{k_0} a_{k_0+1-j,j}\binom{\mu-1}{R_{n-1}-k_0+j-z}\\
&\quad+(-1)^{R_{n-1}}\sum_{j=1}^{k_0-e_n}d_{k_0+p^n+1-R_{n-1}-S_{n-1}-j,j}
\binom{\mu-1}{R_{n-1}+S_{n-1}+j-k_0-z}\\
&=A_0(z)+(-1)^{R_{n-1}}D_0(z).
\end{align*}
We have proved Equation~\ref{Eq1} in the case $k=k_0$.

Step 4: Now we must prove Equation~\ref{Eq1} when $k>k_0$.

Define $z_k \in D_{p^n+R_{n-1}-k}$ by
\begin{align*}
z_k&=\sum_{j=1}^k a_{k+1-j,j}v_{R_{n-1}-k+j,p^n+1-j}\\\
&\quad +(-1)^{R_{n-1}} \sum_{j=1}^{k-e_n}d_{k-e_n+1-j,j}v_{R_{n-1}+S_{n-1}+j-k,p^n-S_{n-1}+1-j},\\
&=\sum_{j=1}^{k_0} (-1)^{k-k_0}a_{k_0+1-j,j}v_{R_{n-1}-k+j,p^n+1-j}\\
&\quad +(-1)^{R_{n-1}} \sum_{j=1}^{k_0-e_n}(-1)^{k-k_0}d_{k_0-e_n+1-j,j}v_{R_{n-1}+S_{n-1}+j-k,p^n-S_{n-1}+1-j}
\end{align*}
applying the relationships between $a_{k+1-j,j}$ and $a_{k_0+1-j,j}$ and $d_{k-e_n+1-j,j}$ and $d_{k_0-e_n+1-j,j}$ from Section~\ref{J}.

We claim that $(g-1)^{\mu-1}(z_k)=x_{p^n-S_{n-1}+k'}$.  What follows is very close to a repetition of the proof of Section~\ref{J}.  The coefficient of $v_{z,p^n-S_{n-1}+k'+1-z}$ in $(g-1)^{\mu-1}(z_k)$ is
\begin{align*}
&(-1)^{k-k_0}\sum_{j=1}^{k_0} a_{k_0+1-j,j}\binom{\mu-1}{R_{n-1}-k+j-z}\\
&\quad+(-1)^{k-k_0}(-1)^{R_{n-1}}\sum_{j=1}^{k_0-e_n}d_{k_0-e_n+1-j,j}\binom{\mu-1}{R_{n-1}+S_{n-1}+j-k-z}\\
&=(-1)^{k-k_0}\sum_{j=1}^{k_0} a_{k_0+1-j,j}\binom{\mu-1}{R_{n-1}-k_0+j-(z+k-k_0)}\\
&\quad+(-1)^{k-k_0}(-1)^{R_{n-1}}\sum_{j=1}^{k_0-e_n}d_{k_0-e_n+1-j,j}\binom{\mu-1}{R_{n-1}+S_{n-1}+j-k_0-(z+k-k_0)}\\
&=(-1)^{k-k_0}\left(A_0(z+k-k_0)+(-1)^{R_{n-1}}D_0(z+k-k_0)\right)\\
&=(-1)^{k-k_0}(-1)^{z+k-k_0-1}\\
&=(-1)^{z-1}.
\end{align*}

We have proved Equation~\ref{Eq1}.
\end{proof}

\begin{lemma}\label{DiAi+1}
When $i p^n+R_{n-1}+1 \leq z \leq (i+1)p^n$, $(-1)^i((-1)^{R_{n-1}}D_i(z)-A_{i+1}(z))=(-1)^{z-1}$.  Since $A_{i+1}(z)=0$ when $z \leq i p^n+p^n-S_{n-1}+k'$ by Lemma~\ref{Reduction}, this implies that $(-1)^i(-1)^{R_{n-1}}D_i(z)=(-1)^{z-1}$ when $i p^n+R_{n-1}+1 \leq z \leq i p^n+p^n-S_{n-1}+k'$.
\end{lemma}
\begin{proof}
By a similar argument to the one stated at the beginning of the proof of Lemma~\ref{AiDi}, it suffices to show 
\begin{equation}\label{Eq2}
(-1)^{R_{n-1}}D_0(z)-A_1(z)=(-1)^{z-1} \qquad \text{ when }R_{n-1}+1  \leq z \leq p^n.
\end{equation}

As we remarked,  $A_{0,p^n,R_{n-1},S_{n-1},k}$ and $D_{0,p^n,R_{n-1},S_{n-1},k}$ would have been better notations for $A_0$ and $D_0$.  Now denote $A_{0,p^n,p^n-R_{n-1},p^n-S_{n-1},k-e_n}$ by $A'_0$ \\and $D_{0,p^n,p^n-R_{n-1},p^n-S_{n-1},k-e_n}$ by $D'_0$.

Switching roles, replace $R_{n-1}$ by $p^n-R_{n-1}$, $S_{n-1}$ by $p^n-S_{n-1}$, and $k$ by $k+p^n-R_{n-1}-S_{n-1}$.  Then $R_{n-1}-k$ becomes $p^n-R_{n-1}-(k+p^n-R_{n-1}-S_{n-1})=S_{n-1}-k$,  and $R_{n-1}+S_{n-1}-k$ becomes 
$(p^n-R_{n-1})+(p^n-S_{n-1})-(k+p^n-R_{n-1}-S_{n-1})=p^n-k$.  Finally $k-e_n+p^n -(p^n-R_{n-1})-(p^n-S_{n-1})=k$.  Then
\[A'_0(a)=\sum_{j=1}^{k-e_n} \delta_{k-e_n+1-j,j}\binom{\mu_k-1}{S_{n-1}-k+j-a}\]
and
\[D'_0(a)=\sum_{j=1}^k \alpha_{k+1-j,j}\binom{\mu_k-1}{p^n-k+j-a},\]
and by Lemma~\ref{AiDi},
\[A'_0(a)+(-1)^{p^n-R_{n-1}}D'_0(a)=(-1)^{a-1}
\]
when $1 \leq a \leq p^n-R_{n-1}$.   Recall
\[D_0(a)=\sum_{j=1}^{k-e_n} \delta_{k-e_n+1-j,j}\binom{\mu_k-1}{R_{n-1}+S_{n-1}-k+j-a}\]
and
\[A_1(a)=\sum_{j=1}^k \alpha_{k+1-j,j}\binom{\mu_k-1}{p^n+R_{n-1}-k+j-a}.\]
Hence $D_0(a)=A'_0(a-R_{n-1})$ and $A_1=D'_0(a-R_{n-1})$.

If $R_{n-1}+1 \leq a \leq p^n$, then $1 \leq a-R_{n-1} \leq p^n-R_{n-1}$ and $A'_0(a-R_{n-1})+(-1)^{p^n-R_{n-1}}D'_0(a-R_{n-1})=(-1)^{a-R_{n-1}-1}$.   Then
\begin{align*}
(-1)^{R_{n-1}}&D_0(a)+(-1)^{p^n}A_1(a)\\
&=(-1)^{R_{n-1}}\left(A'_0(a-R_{n-1})+(-1)^{p^n-R_{n-1}}D'_0(a-R_{n-1})\right)\\
&=(-1)^{R_{n-1}}(-1)^{a-R_{n-1}-1}\\
&=(-1)^{a-1}.
\end{align*}
We have proved Equation~\ref{Eq2}.
\end{proof}

\section{$z_\ell$ and $y_\ell$ when $m_{n-1}>0$}\label{m>0}
Recall that we are assuming $ \max(r+s-p^{n+1},0)+1 \leq \ell \leq \min(r,s)$.  Write $\ell$ as $t p^n+k$ where $t$ and $k$ are  integers with $t \geq 0$ and $1 \leq k \leq p^n$.    Then
\[ \max(r+s-p^{n+1},0)+1 \leq t p^n+k \leq \min(r_np^n+R_{n-1},s_n p^n+S_{n-1}).\]
In Lemmas~\ref{M1} to ~\ref{M5}, we will be interested in the invertibility of certain  matrices of size $(t+1)\times(t+1)$ or $t \times t$ whose entries are binomial coefficients where $k$ is restricted to certain subintervals of $[1,p^n]$.

We will use these matrices in working our way through the five cases of Corollary~\ref{RenaudCor} in Propositions~\ref{R1} to~\ref{R5} with the third case of $m_{n-1}+1 \leq k \leq M_{n-1}$ spanning Propositions~\ref{R3AndR<S} and~\ref{R3AndR>S}.  Note that only the second and fourth cases require an appeal to recursion.
  
Also $e_n=m_{n-1}+M_{n-1}-p^n$ and $X$ will the $(t+1) \times 1$ matrix over $F$ with $(i,1)$ entry $(-1)^{i-1}$.

\begin{lemma}\label{M1}
Suppose that $M(1)$ is $(t+1) \times (t+1)$ with $(i,j)$ entry $\binom{r_n+s_n-2 t}{r_n-t+j-i}$.  Let $1 \leq k \leq m_{n-1}$.  Then $M(1)$ is invertible and $r_n+s_n-2t<p$.
\end{lemma}
\begin{proof}
By Lemma~\ref{pdiv} it suffices to show that $r_n-t \leq r_n+s_n-2t$ and $(r_n+s_n-2t)+t=r_n+s_n-t<p$.

Since $t p^n +k \leq \min(r_np^n+R_{n-1},s_n p^n+S_{n-1})$, $t \leq \min(r_n,s_n)$. Then $r_n-t \leq r_n-t+s_n-t=r_n+s_n-2t$.
 
Since $\ell=t p^n+k>r+s-p^{n+1}=(r_n+s_n-p)p^n+m_{n-1}+M_{n-1}$ and $k \leq m_{n-1}+M_{n-1}$, $t>r_n+s_n-p$.  Thus $p>r_n+s_n-t$.  Hence $r_n+s_n-2t \leq r_n+s_n-t<p$.
\end{proof}

\begin{proposition}\label{R1}
Assume that $t \leq \min(r_n,s_n)$ and $1 \leq k \leq \max(e_n,0)$.  Then $\lambda_\ell=(r_n+s_n-2t)p^n+p^n$ by Corollary~\ref{RenaudCor}.   So $D_{(r_n+s_n-t)p^n+R_{n-1}+S_{n-1}-k}$  is mapped by $(g-1)^{(r_n+s_n-2t)p^n}$ to $D_{t p^n+R_{n-1}+S_{n-1}-k}$.  
\begin{enumerate}
\item Define $z_\ell \in D_{ t p^n+R_{n-1}+S_{n-1}-k}$ by
\[z_\ell=\sum_{i=0}^t (-1)^i (-1)^{R_{n-1}-k} v_{i p^n+R_{n-1}+1-k,(t-i)p^n+S_{n-1}}.
\]
Then $(g-1)^{p^n-1}(z_\ell)=x_{(t-1)p^n+R_{n-1}+S_{n-1}+1-k}$. 

\item With $\zeta=M(1)^{-1}X$, define $y_\ell \in D_{(r_n+s_n-t)p^n+R_{n-1}+S_{n-1}-k}$ by
\[
y_\ell=\sum_{i=0}^t \zeta_{i+1,1}(-1)^{R_{n-1}-k}v_{(r_n-t+i)p^n+R_{n-1}+1-k,(s_n-i)p^n+S_{n-1}}.
\]
Then $(g-1)^{(r_n+s_n-2t)p^n}(y_\ell)=z_\ell$.
\end{enumerate}
\end{proposition}
Note that
\[z_\ell=\sum_{i=0}^t (-1)^i \sum_{j=1}^k a_{k+1-j,j}v_{i p^n+R_{n-1}+j-k,(t-i)p^n+S_{n-1}+1-j}\]
with a corresponding expression for $y_\ell$
since the $k$th anti-diagonal of $A=J(R_{n-1},S_{n-1})$ is $((-1)^{R_{n-1}-k},0,\dots,0)$, but we do not need to make that recursive call here.
\begin{proof}
(1) Recall that $\mathcal{B}_{(t-1)p^n+R_{n-1}+S_{n-1}+1-k}$ is
\[\{v_{a, (t-1)p^n+R_{n-1}+S_{n-1}+2-k-a} \mid 1 \leq a \leq (t-1)p^n+R_{n-1}+S_{n-1}+1-k\}.
\]
Now the coefficient $\alpha_a$ of $v_{a, (t-1)p^n+R_{n-1}+S_{n-1}+2-k-a}$ in $(g-1)^{p^n-1}(z_\ell)$ is
\[\alpha_a=\sum_{i=0}^t  (-1)^{i+R_{n-1}-k} \binom{p^n-1}{i p^n+R_{n-1}+1-k-a}.\]
For $1  \leq j \leq t$, define the interval $I_j=[(j-1)p^n +R_{n-1}+2-k,j p^n+R_{n-1}+1-k]$ and define $I_0=[1,R_{n-1}+1-k]$.
Since $1 \leq a \leq (t-1)p^n+m_{n-1}+M_{n-1}+1-k$, it follows that $a$ falls in exactly one of the intervals $I_j$ with $0 \leq j\leq t$.
Suppose that $a \in I_i$.  Then $0 \leq i p^n+R_{n-1}+1-k-a \leq p^n-1$ and $ (-1)^{i+R_{n-1}-k} \binom{p^n-1}{i p^n+R_{n-1}+1-k-a}$ is the only nonzero term in $\alpha_a$.  Hence
\begin{align*}
\alpha_a&=(-1)^{i+R_{n-1}-k} \binom{p^n-1}{i p^n+R_{n-1}+1-k-a}\\
&=(-1)^{i+R_{n-1}-k+i p^n+R_{n-1}+1-k-a}\\
&=(-1)^{i+i p^n+1-a}.\\
\end{align*}
If $p=2$, then $(-1)^{i +i p^n+1-a}=1=(-1)^{a-1}$.  If $p$ is odd, then $i+ip^n$ is even and $(-1)^{i+i p^n+1-a}=(-1)^{a-1}$.
We have proved $(g-1)^{p^n-1}(z_\ell)=x_{(t-1)p^n+R_{n-1}+S_{n-1}+1-k}$.

(2) When $1 \leq i \leq t+1$ the $(i,1)$ entry of $M(1)\zeta$ is
\[(-1)^{i-1}=\sum_{j=1}^{t+1} \binom{r_n+s_n-2t}{r_n-t+j-i}\zeta_{j,1}=\sum_{j=0}^{t} \binom{r_n+s_n-2t}{r_n-t+j+1-i}\zeta_{j+1,1}.\]
By Corollary~\ref{p^n2}, since the only vectors in $\mathcal{B}_{t p^n+R_{n-1}+S_{n-1}-k}$ with possibly nonzero coefficients in 
$(g-1)^{(r_n+s_n-2t)p^n}(y_\ell)$ are $v_{a p^n+R_{n-1}+1-k,(t-a)p^n+S_{n-1}}$ with $0 \leq a \leq t$,
\begin{align*}
(g-1)&^{(r_n+s_n-2t)p^n}(v_{(r_n-t+i)p^n+R_{n-1}+1-k,(s_n-i)p^n+S_{n-1}})\\
&=\sum_{a=0}^t \binom{r_n+s_n-2t}{r_n-t+i-a}v_{a p^n+R_{n-1}+1-k,(t-a)p^n+S_{n-1}},
\end{align*}
using the fact that $r_n+s_n-2t<p$ from Lemma~\ref{M1}.

Hence
\begin{align*}
(g-1)&^{(r_n+s_n-2t)p^n}(y_\ell)\\
&=\sum_{i=0}^t \zeta_{i+1,1}(-1)^{R_{n-1}-k}
\left(\sum_{a=0}^t \binom{r_n+s_n-2t}{r_n-t+i-a}v_{a p^n+R_{n-1}+1-k,(t-a)p^n+S_{n-1}}\right)\\
&=(-1)^{R_{n-1}-k}\sum_{a=0}^t\left(\sum_{i=0}^t \binom{r_n+s_n-2t}{r_n-t+i-a}\zeta_{i+1,1}\right)
v_{a p^n+R_{n-1}+1-k,(t-a)p^n+S_{n-1}}\\
&=(-1)^{R_{n-1}-k}\sum_{a=0}^t(-1)^a v_{a p^n+R_{n-1}+1-k,(t-a)p^n+S_{n-1}}\\
&=z_\ell.
\end{align*}
\end{proof}

\begin{lemma}\label{M2}
Suppose that $M(2)$ is the $t \times t$ with $(i,j)$ entry $\binom{r_n+s_n-2 t}{r_n-t+j-i}$ and $1 \leq k \leq m_{n-1}$.  Then $M(2)$ is invertible and $r_n+s_n-2t<p$.
\end{lemma}
\begin{proof}
Similar to the proof of Lemma~\ref{M1} except that we need to show that $r_n+s_n-2t+t-1=r_n+s_n-t-1<p$.  Since $t p^n+k>(r_n+s_n-p)p^n+m_{n-1}+M_{n-1}$ and $k \leq m_{n-1}+M_{n-1}$, $t>r_n+s_n-p$, so $p>r_n+s_n-t$.  Again $r_n+s_n-2t \leq r_n+s-n-t<p$.
\end{proof}

Let $Y$ be the $t \times 1$ matrix whose $(i,1)$ entry is $(-1)^{i-1}$ while $X$ is the $(t+1)\times 1$ matrix whose $(i,1)$ entry is $(-1)^{i-1}$.

\begin{proposition}\label{R2}
Assume that $t \leq \min(r_n,s_n)$ and $\max(e_n,0)+1 \leq k \leq m_{n-1}$.  Then $\lambda_\ell=(r_n+s_n-2t)p^n+\mu_k$  by Corollary~\ref{RenaudCor}.  Hence  $(g-1)^{(r_n+s_n-2t)p^n}$ maps $D_{(r_n+s_n-t)p^n+R_{n-1}+S_{n-1}-k}$ to $D_{t p^n+R_{n-1}+S_{n-1}-k}$.   Assume that $A=J(R_{n-1},S_{n-1})$ and $D=J(p^n-R_{n-1},p^n-S_{n-1})$ are known. 
\begin{enumerate}
\item  Define $z_\ell \in D_{ t p^n+R_{n-1}+S_{n-1}-k}$ by
\begin{align*}
z_\ell&=\sum_{i=0}^t(-1)^i\sum_{j=1}^k a_{k+1-j,j}v_{i p^n+R_{n-1}-k+j,(t-i)p^n+S_{n-1}+1-j}\\
&\quad+(-1)^{R_{n-1}}\sum_{i=0}^{t-1}(-1)^i \sum_{j=1}^{k-e_n} d_{k-e_n+1-j,j}v_{i p^n+R_{n-1}+S_{n-1}+j-k,(t-i)p^n+1-j}.
\end{align*}
Then $(g-1)^{\mu_k-1}(z_\ell)=x_{t p^n+k'}$. 
\item With $\zeta=M(1)^{-1} X$ and $\eta=M(2)^{-1}Y$,
define $y_\ell \in D_{(r_n+s_n-t)p^n+R_{n-1}+S_{n-1}-k}$ by
\begin{align*}
y_\ell&=\sum_{i=0}^t \zeta_{i+1,1} \sum_{j=1}^k a_{k+1-j,j} v_{(r_n-t+i)p^n +R_{n-1}+j-k,(s_n-i)p^n+S_{n-1}+1-j}\\
&\quad + (-1)^{R_{n-1}} \sum_{i=0}^{t-1} \eta_{i+1,1} \sum_{j=1}^{k-e_n}d_{k-e_n+1-j,j}v_{(r_n-t+i)p^n+m_{n-1}+M_{n-1}+j-k,(s_n-i)p^n+1-j}
\end{align*}
Then $(g-1)^{(r_n+s_n-2t)p^n}(y_\ell)=z_\ell$.
\end{enumerate}
\end{proposition}
\begin{proof}
(1) The coefficient of $v_{z,(t-1) p^n+k'+1-z}$ in $(g-1)^{\mu_k-1}(z_\ell)$ is
\begin{align*}
&\sum_{i=0}^t(-1)^i\sum_{j=1}^k a_{k+1-j,j}\binom{\mu_k-1}{i p^n+R_{n-1}-k+j-z}\\
&\quad +(-1)^{R_{n-1}}\sum_{i=0}^{t-1}(-1)^i \sum_{j=1}^{k-e_n}d_{k-e_n+1-j,j}
\binom{\mu_k-1}{i p^n+R_{n-1}+S_{n-1}+j-k-z}.
\end{align*}
But this equals $f_t(z)=(-1)^{z-1}$ by Theorem~\ref{ZThm} when $1 \leq z \leq t p^n+k'$. 
We have proved that $(g-1)^{\mu_k-1}(z_\ell)=x_{t p^n+k'}$. 

(2) Let
\[z'_\ell=\sum_{i=0}^t(-1)^i\sum_{j=1}^k a_{k+1-j,j}v_{i p^n+R_{n-1}-k+j,(t-i)p^n+S_{n-1}+1-j},\]
\[z''_\ell=\sum_{i=0}^{t-1}(-1)^i \sum_{j=1}^{k-e_n}d_{k-e_n+1-j,j}v_{i p^n+R_{n-1}+S_{n-1}+j-k,(t-i)p^n+1-j},\]
\[y'_\ell=\sum_{i=0}^t \zeta_{i,0} \sum_{j=1}^k a_{k+1-j,j} v_{(r_n-t+i)p^n +R_{n-1}+j-k,(s_n-i)p^n+S_{n-1}+1-j},\]
and 
\[y''_\ell=\sum_{i=0}^{t-1} \eta_{i,0} \sum_{j=1}^{k-e_n}d_{k-e_n+1-j,j}v_{(r_n-t+i)p^n+m_{n-1}+M_{n-1}+j-k,(s_n-i)p^n+1-j},\]
so $z_\ell=z'_\ell+(-1)^{R_{n-1}} z''_\ell$ and $y_\ell=y'_\ell+(-1)^{R_{n-1}}y''_\ell$.
We will prove our result by showing that
$(g-1)^{(r_n+s_n-2t)p^n}(y'_\ell)=z'_\ell$
and
$(g-1)^{(r_n+s_n-2t)p^n}(y''_\ell)=z''_\ell$.

By Corollary~\ref{p^n2}, since the only vectors in $\mathcal{B}_{t p^n+R_{n-1}+S_{n-1}-k}$ with possibly nonzero coefficients in 
$(g-1)^{(r_n+s_n-2t)p^n}(v_{(r_n-t+i)p^n +R_{n-1}+j-k,(s_n-i)p^n+S_{n-1}+1-j})$ are \\$v_{ap^n+R_{n-1}-k+j,(t-a)p^n+S_{n-1}+1-j}$ with $0 \leq a \leq t$,
\begin{align*}
(g-1)&^{(r_n+s_n-2t)p^n}( v_{(r_n-t+i)p^n +R_{n-1}+j-k,(s_n-i)p^n+S_{n-1}+1-j})\\
&=\sum_{a=0}^t \binom{r_n-s_n-2t}{r-t+i-a}v_{a p^n+R_{n-1}-k+j,(t-a)p^n+S_{n-1}+1-j},
\end{align*}
using the fact that $r_n-s_n-2t<p$ by Lemma~\ref{M1}.
When $1 \leq i \leq t+1$ the $(i,1)$ entry of $M(1) \zeta$ is
\[(-1)^{i-1}=\sum_{j=1}^{t+1}  \binom{r_n+s_n-2t}{r_n-t+j-i}\zeta_{j,1}= \sum_{j=0}^{t}  \binom{r_n+s_n-2t}{r_n-t+j+1-i}\zeta_{j+1,1}.\]
Hence
\begin{align*}
(g-1)&^{(r_n+s_n-2t)p^n}(y'_\ell)\\
&=\sum_{i=0}^t \zeta_{i+1,1} \sum_{j=1}^k a_{k+1-j,j}\sum_{a=0}^t \binom{r_n-s_n-2t}{r-t+i-a}v_{a p^n+R_{n-1}-k+j,(t-a)p^n+S_{n-1}+1-j}\\
&=\sum_{a=0}^t \left(\sum_{i=0}^t\binom{r_n-s_n-2t}{r-t+i-a}\zeta_{i+1,1}\right)\sum_{j=1}^k a_{k+1-j,j}v_{a p^n+R_{n-1}-k+j,(t-a)p^n+S_{n-1}+1-j}\\
&=\sum_{a=0}^t(-1)^a \sum_{j=1}^k a_{k+1-j,j}v_{a p^n+R_{n-1}-k+j,(t-a)p^n+S_{n-1}+1-j}\\
&=z'_\ell.
\end{align*}
Again by Corollary~\ref{p^n2}, since the only vectors in $\mathcal{B}_{t p^n+R_{n-1}+S_{n-1}-k}$ with possibly nonzero coefficients in 
$(g-1)^{(r_n+s_n-2t)p^n}(v_{(r_n-t+i)p^n+m_{n-1}+M_{n-1}+j-k,(s_n-i)p^n+1-j})$ are \\$v_{ap^n+m_{n-1}+M_{n-1}+j-k,(t-a)p^n+1-j}$ with $0 \leq a \leq t-1$,
\begin{align*}
(g-1)&^{(r_n+s_n-2t)p^n}(v_{(r_n-t+i)p^n+m_{n-1}+M_{n-1}+j-k,(s_n-i)p^n+1-j})\\
&=\sum_{a=0}^{t-1}\binom{r_n+s_n-2t}{r_n-t+i-a}v_{ap^n+m_{n-1}+M_{n-1}+j-k,(t-a)p^n+1-j},
\end{align*}
using the fact that  $r_n+s_n-2t<p$ by Lemma~\ref{M2}.

When $1 \leq i \leq t$ the $(i,1)$ entry of $M(2) \eta$ is
\[(-1)^{i-1}=\sum_{j=1}^t \binom{r_n+s_n-2t}{r_n-t+j-i}\eta_{j,1}=\sum_{j=0}^{t-1} \binom{r_n+s_n-2t}{r_n-t+j+1-i}\eta_{j+1,1}.\]
Hence
\begin{align*}
(g-1)&^{(r_n+s_n-2t)p^n}(y''_\ell)\\
&=\sum_{i=0}^{t-1} \eta_{i+1,1} \sum_{j=1}^{k-e_n}d_{k-e_n+1-j,j}\sum_{a=0}^{t-1}\binom{r_n+s_n-2t}{r_n-t+i-a}v_{ap^n+m_{n-1}+M_{n-1}+j-k,(t-a)p^n+1-j}\\
&=\sum_{a=0}^{t-1}\left(\sum_{i=0}^{t-1} \binom{r_n+s_n-2t}{r_n-t+i-a}\eta_{i+1,1}\right)\sum_{j=1}^{k-e_n}d_{k-e_n+1-j,j}
v_{ap^n+m_{n-1}+M_{n-1}+j-k,(t-a)p^n+1-j}\\
&=\sum_{a=0}^{t-1}(-1)^a\sum_{j=1}^{k-e_n}d_{k-e_n+1-j,j}
v_{ap^n+m_{n-1}+M_{n-1}+j-k,(t-a)p^n+1-j}\\
&=z''_\ell.
\end{align*}
We have proved that $(g-1)^{(r_n+s_n-2t)p^n}(y_\ell)=z_\ell$.
\end{proof}

\begin{lemma}\label{M3}
Suppose that $M(3)$ is $(t+1) \times (t+1)$ with $(i,j)$ entry $\binom{r_n+s_n-2 t-1}{r_n-t+j-i-1}$ and that either $m_{n-1}+1 \leq k\leq M_{n-1}$ or 
$M_{n-1}+1 \leq k \leq p^n-e_2$.  Then $M(3)$ is invertible and $r_n+s_n-2 t-1<p$.
\end{lemma}
\begin{proof}
By Lemma~\ref{pdiv} it suffices to show that $r_n-t-1 \leq r_n+s_n-2t-1$ and $(r_n+s_n-2t-1)+t=r_n+s_n-t-1<p$.

Since  $t p^n +k \leq \min(r_np^n+R_{n-1},s_n p^n+S_{n-1})$, $t \leq \min(r_n,s_n)$.   Thus $r_n-t-1 \leq (r_n-t-1)+(s_n-t)=r_n+s_n-2t-1$.

Assume that $m_{n-1}+1 \leq k\leq M_{n-1}$.  Since $tp^n+k>(r_n+s_n-p)p^n+m_{n-1}+M_{n-1}$ and $k \leq m_{n-1}+M_{n-1}$, $t>r_n+s_n-p$, from which it follows that $p>r_n+s_n-t>r_n+s_n-t-1$.  It follows that $r_n+s_n-2t-1<p$.

Assume that $M_{n-1}+1 \leq k \leq p^n-e_2$.  Then $p^n-e_2$ equals either $m_{n-1}+M_{n-1}$ if $m_{n-1}+M_{n-1} \leq p^n$ or $p^n$ if $m_{n-1}+M_{n-1}>p^n$.  In both cases $k \leq m_{n-1}+M_{n-1}$, and the argument of the previous paragraph holds.   Again it follows that $r_n+s_n-2t-1<p$.
\end{proof}

Now we look at the situation when $m_{n-1}+1 \leq k \leq M_{n-1}$.  Since this is empty when $m_{n-1}=M_{n-1}$, we assume that $m_{n-1}<M_{n-1}$.  This breaks into two cases: $R_{n-1}<S_{n-1}$ and $R_{n-1}>S_{n-1}$.

\begin{proposition}\label{R3AndR<S}
Assume that $R_{n-1}<S_{n-1}$, $t \leq \min(r_n-1,s_n)$, and $R_{n-1}+1 \leq k \leq S_{n-1}$.   Then $\lambda_\ell=(r_n+s_n-2t)p^n$ by Corollary~\ref{RenaudCor}.  Hence $(g-1)^{(r_n+s_n-2t-1)p^n}$ maps $D_{(r_n+s_n-t)p^n+R_{n-1}+S_{n-1}-k}$ to $D_{(t+1)p^n+R_{n-1}+S_{n-1}-k}$.
\begin{enumerate}
\item Define $z_\ell \in D_{(t+1)p^n+R_{n-1}+S_{n-1}-k}$ by
\[z_\ell=(-1)^{R_{n-1}}\sum_{i=0}^t (-1)^{i+k-1}v_{(i+1)p^n +R_{n-1}+1-k,(t-i)p^n +S_{n-1}}\]
Then $(g-1)^{p^n-1}(z_\ell)=x_{t p^n+R_{n-1}+S_{n-1}+1-k}$.
\item Let $\zeta=M(3)^{-1}X$.  Define $y_\ell \in  D_{(r_n+s_n-t)p^n+R_{n-1}+S_{n-1}-k}$ by
\[y_\ell=(-1)^{R_{n-1}+k-1}\sum_{i=0}^t \zeta_{i+1,1}v_{(r_n-t+i)p^n +R_{n-1}+1-k,(s_n-i)p^n+S_{n-1}}.
\]
Then $(g-1)^{(r_n+s_n-2 t-1)p^n}(y_\ell)=z_\ell$.
\end{enumerate}
\end{proposition}
Note that
\[z_\ell=(-1)^{R_{n-1}}\sum_{i=0}^t (-1)^i \sum_{j=1}^{k-R_{n-1}} b_{k-R_{n-1}+1-j,j}
v_{(i+1)p^n +R_{n-1}+j-k,(t-i)p^n +S_{n-1}+1-j}\]
with a corresponding expression for $y_\ell$ since the $k'$th anti-diagonal of $B=J(p^n-R_{n-1},S_{n-1})$, with $1 \leq k' \leq S_{n-1}-R_{n-1}$,  is $((-1)^{R_{n-1}+k'-1},0,\dots,0)$, so $b_{k-R_{n-1},1}=(-1)^{k-1}$, but we do not need to make that recursive call here.
\begin{proof}
(1) Recall that $\mathcal{B}_{t p^n+R_{n-1}+S_{n-1}+1-k}$ consists of
$v_{a, t p^n+R_{n-1}+S_{n-1}+1-k-a}$ where $1 \leq a \leq t p^n+R_{n-1}+S_{n-1}+1-k$.
The coefficient $\alpha_a$ of $v_{a,tp^n+R_{n-1}+S_{n-1}+2-k-a}$ in $(g-1)^{p^n-1}(z_\ell)$ is
\[(-1)^{R_{n-1}}\sum_{i=0}^t (-1)^{i +k-1}\binom{p^n-1}{(i+1)p^n +R_{n-1}+1-k-a}.\]
By a now familiar argument, there is only one $i$ such that $\binom{p^n-1}{(i+1)p^n +R_{n-1}+1-k-a} \neq 0$.
Hence
\begin{align*}
\alpha_a&=(-1)^{R_{n-1}+i+k-1} \binom{p^n-1}{(i+1)p^n +R_{n-1}+1-k-a}\\
&=(-1)^{R_{n-1}+i+k-1}(-1)^{(i+1)p^n +R_{n-1}+1-k-a}\\
&=(-1)^{p^n-a}\\
&=(-1)^{a-1}.
\end{align*}
This proves $(g-1)^{p^n-1}(z_\ell)=x_{t p^n+R_{n-1}+S_{n-1}+1-k}$.

(2) When $1 \leq i \leq t+1$ the $(i,1)$ entry of $M(3) \zeta$ is
\[(-1)^{i-1}=\sum_{j=1}^{t+1}\binom{r_n+s_n-2 t-1}{r_n-t+j-i-1}\zeta_{j,1}= \sum_{j=0}^t \binom{r_n+s_n-2 t-1}{r_n-t+j-i}\zeta_{j+1,1}.\]
By Corollary~\ref{p^n2}, since the only vectors in $\mathcal{B}_{t p^n+R_{n-1}+S_{n-1}-k}$ with possibly nonzero coefficients in 
$(g-1)^{(r_n+s_n-2 t-1)p^n}(v_{(r_n-t+i)p^n +R_{n-1}+1-k,(s_n-i)p^n+S_{n-1}})$ are \\$v_{(a+1)p^n+R_{n-1}+1-k,(t-a)p^n+S_{n-1}}$ with $0 \leq a \leq t$,
\begin{align*}
(g-1)&^{(r_n+s_n-2 t-1)p^n}(v_{(r_n-t+i)p^n +R_{n-1}+1-k,(s_n-i)p^n+S_{n-1}}\\
&=\sum_{a=0}^t \binom{r_n+s_n-2 t-1}{r_n-t+i-a-1}v_{(a+1)p^n+R_{n-1}+1-k,(t-a)p^n+S_{n-1}},
\end{align*}
using the fact that $r_n+s_n-2t-1<p$ by Lemma~\ref{M3}.

Hence
\begin{align*}
(g-1)&^{(r_n+s_n-2 t-1)p^n}(y_\ell)\\
&=(-1)^{R_{n-1}+k-1}\sum_{i=0}^t \zeta_{i+1,1}
\left(\sum_{a=0}^t \binom{r_n+s_n-2 t-1}{r_n-t+i-a-1}v_{(a+1)p^n+R_{n-1}+1-k,(t-a)p^n+S_{n-1}}\right)\\
&=(-1)^{R_{n-1}+k-1}\sum_{a=0}^t
\left(\sum_{i=0}^t \binom{r_n+s_n-2 t-1}{r_n-t+i-a-1}\zeta_{i+1,1}\right)
v_{(a+1)p^n+R_{n-1}+1-k,(t-a)p^n+S_{n-1}}\\
&=(-1)^{R_{n-1}+k-1}\sum_{a=0}^t (-1)^a v_{(a+1)p^n+R_{n-1}+1-k,(t-a)p^n+S_{n-1}}\\
&=z_\ell.
\end{align*}
\end{proof}

\begin{lemma}\label{M4}
Suppose that $M(4)$ is $(t+1) \times (t+1)$ with $(i,j)$ entry  $\binom{r_n+s_n-2 t-1}{r_n-t+j-i}$ and that either $S_{n-1}+1 \leq k\leq R_{n-1}$ or 
$M_{n-1}+1 \leq k \leq p^n-e_2$.  Then $M(4)$ is invertible and $r_n+s_n-2 t-1<p$.
\end{lemma}
\begin{proof}
By Lemma~\ref{pdiv} it suffices to show that $r_n-t \leq r_n+s_n-2t-1$ and $(r_n+s_n-2t-1)+t=r_n+s_n-t-1<p$.  

Since $ t p^n+k \leq s_np^n+S_{n-1}$  and $k >S_{n-1}$ in both cases, $t \leq s_n-1$.  Thus $r_n-t \leq (r_n-t)+(s_n-t-1)=r_n+s_n-2t-1$.

In both cases $k \leq m_{n-1}+M_{n-1}$.  Since $t p^n+k >(r_n+s_n-p)p^n+m_{n-1}+M_{n-1}$, $t>r_n+s_n-p$.  Hence $p>r_n+s_n-t>r_n+s_n-t-1$.   It follows that $r_n+s_n-2 t-1<p$.
\end{proof}
\begin{proposition}\label{R3AndR>S}
Assume that $R_{n-1}>S_{n-1}$, $S_{n-1}+1 \leq k \leq R_{n-1}$,  and $t \leq \min(r_n,s_n-1)$.  Then $\lambda_\ell=
(r_n+s_n-2t)p^n$ by Corollary~\ref{RenaudCor}.  Hence $(g-1)^{(r_n+s_n-2t-1)p^n}$ maps $D_{(r_n+s_n-t)p^n+R_{n-1}+S_{n-1}-k}$ to $D_{(t+1)p^n+R_{n-1}+S_{n-1}-k}$.
\begin{enumerate}
\item Define $z_\ell \in D_{(t+1)p^n+R_{n-1}+S_{n-1}-k}$ by
\[z_\ell=(-1)^{R_{n-1}+S_{n-1}-k}\sum_{i=0}^t (-1)^i v_{i p^n+R_{n-1}+S_{n-1}+1-k,(t-i+1)p^n}
\]
The $(g-1)^{p^n-1}(z_\ell)=x_{tp^n+R_{n-1}+S_{n-1}+1-k}$.
\item Let $\zeta=M(4)^{-1}X$.  Define $y_\ell \in  D_{(r_n+s_n-t)p^n+R_{n-1}+S_{n-1}-k}$ by
\[y_\ell=(-1)^{R_{n-1}+S_{n-1}-k}\sum_{i=0}^t \zeta_{i+1,1}v_{(r_n-t+i)p^n+R_{n-1}+S_{n-1}+1-k,(s_n-t)p^n}.
\]
Then $(g-1)^{(r_n+s_n-2t-1)p^n}(y_\ell)=z_\ell$.
\end{enumerate}
\end{proposition}
Note that
\[z_\ell=\sum_{i=0}^t (-1)^i \sum_{j=1}^{k-S_{n-1}}c_{k-S_{n-1}+1-j,j}v_{i p^n+R_{n-1}+S_{n-1}+j-k,(t-i+1)p^n+1-j}
\]
with a corresponding expression for $y_\ell$ since the $k'$th anti-diagonal of $C=J(R_{n-1},p^n-S_{n-1})$, with $1 \leq k' \leq R_{n-1}-S_{n-1}$,  is $((-1)^{R_{n-1}-k'},0,\dots,0)$, so $c_{k-S_{n-1},1}=(-1)^{R_{n-1}+S_{n-1}-k}$, but we do not need to make that recursive call here.
\begin{proof}
(1) Here $\lambda_\ell=(r_n+s_n-2t)p^n$.  Recall that  $\mathcal{B}_{tp^n+R_{n-1}+S_{n-1}+1-k}$ consists of
$v_{a, tp^n+R_{n-1}+S_{n-1}+2-k-a}$ where $1 \leq a \leq tp^n+R_{n-1}+S_{n-1}+1-k$.
The coefficient $\alpha_a$ of $v_{a,tp^n+R_{n-1}+S_{n-1}+2-k-a}$ in $(g-1)^{(p^n-1}(z_\ell)$ is
\[(-1)^{R_{n-1}+S_{n-1}-k}\sum_{i=0}^t(-1)^i \binom{p^n-1}{i p^n+R_{n-1}+S_{n-1}+1-k-a}.\]
By a now familiar argument there is only one $i$ such that $\binom{p^n-1}{i p^n+R_{n-1}+S_{n-1}+1-k-a} \neq 0$.
Hence
\begin{align*}
\alpha_a&=(-1)^{R_{n-1}+S_{n-1}-k+i} \binom{p^n-1}{i p^n+R_{n-1}+S_{n-1}+1-k-a}\\
&=
(-1)^{R_{n-1}+S_{n-1}-k+i}(-1)^{i p^n+R_{n-1}+S_{n-1}+1-k-a}\\
&=(-1)^{a-1}.
\end{align*}
We have proved $(g-1)^{p^n-1}(z_\ell)=x_{tp^n+R_{n-1}+S_{n-1}+1-k}$.

(2) By Corollary~\ref{p^n2}, since the only vectors in $\mathcal{B}_{tp^n+R_{n-1}+S_{n-1}+1-k}$ with nonzero coefficients in
$(g-1)^{(r_n+s_n-2t-1)p^n}(v_{(r_n-t+i)p^n+R_{n-1}+S_{n-1}+1-k,(s_n-t)p^n})$ are \\$ v_{a p^n+R_{n-1}+S_{n-1}+1-k,(t-a+1)p^n}$ with $0 \leq a \leq t$,
\begin{align*}
(g-1)&^{(r_n+s_n-2t-1)p^n}(v_{(r_n-t+i)p^n+R_{n-1}+S_{n-1}+1-k,(s_n-t)p^n})\\
&=\sum_{a=0}^t \binom{r_n+s_n-2t-1}{r_n-t+i-a} v_{a p^n+R_{n-1}+S_{n-1}+1-k,(t-a+1)p^n},
\end{align*}
using the fact that $r_n+s_n-2t-1<p$ by Lemma~\ref{M4}.

When $1 \leq i \leq t+1$ the $(i,1)$ entry of $M(4) \zeta$ is
\[(-1)^{i-1}=\sum_{j=1}^{t+1}\binom{r_n+s_n-2 t-1}{r_n-t+j-i}\zeta_{j,1}=\sum_{j=0}^t \binom{r_n+s_n-2 t-1}{r_n-t+j+1-i}\zeta_{j+1,1}.
\]
Then
\begin{align*}
(g-1)&^{(r_n+s_n-2t-1)p^n}(y_\ell)\\
&=(-1)^{R_{n-1}+S_{n-1}-k}\sum_{i=0}^t \zeta_{i+1,1}
\left(\sum_{a=0}^t \binom{r_n+s_n-2t-1}{r_n-t+i-a} v_{a p^n+R_{n-1}+S_{n-1}+1-k,(t-a+1)p^n}\right)\\
&=(-1)^{R_{n-1}+S_{n-1}-k}\sum_{a=0}^t 
\left(\sum_{i=0}^t \binom{r_n+s_n-2t-1}{r_n-t+i-a}\zeta_{i+1,1}\right)
v_{a p^n+R_{n-1}+S_{n-1}+1-k,(t-a+1)p^n}\\
&=(-1)^{R_{n-1}+S_{n-1}-k}\sum_{a=0}^t  (-1)^a v_{a p^n+R_{n-1}+S_{n-1}+1-k,(t-a+1)p^n}\\
&=z_\ell.
\end{align*}
We have proved $(g-1)^{(r_n+s_n-2t-1)p^n}(y_\ell)=z_\ell$.
\end{proof}

In the next two propositions $p^n-\max(-e_n,0)=m_{n-1}+M_{n-1}$ if $m_{n-1}+M_{n-1} \leq p^n$ and equals $p^n$ otherwise.  First we assemble some results for Proposition~\ref{R4}.

\begin{lemma}\label{R4aux}
When $M_{n-1}+1 \leq k \leq p^n-\max(-e_n,0)$,
\begin{enumerate}
\item $\mu_{m_{n-1}+M_{n-1}+1-k}=k_0+k_\infty-1-m_{n-1}-M_{n-1}$,
\item $p^n-\mu_{m_{n-1}+M_{n-1}+1-k}-1=p^n+m_{n-1}+M_{n-1}-k_0-k_\infty$,
\item $R_{n-1}+S_{n-1}+1-k+\mu_{m_{n-1}+M_{n-1}+1-k}=k'$.
\end{enumerate}
\end{lemma}
\begin{proof}
Now
\begin{align*}
\mu_{m_{n-1}+M_{n-1}+1-k}&=m_{n-1}+M_{n-1}+1-(m_{n-1}+M_{n-1}+1-k)_0\\
&\quad-(m_{n-1}+M_{n-1}+1-k)_\infty\\
&=k_0+k_\infty-1-m_{n-1}-M_{n-1},
\end{align*}
\[p^n-\mu_{m_{n-1}+M_{n-1}+1-k}-1=p^n+m_{n-1}+M_{n-1}-k_0-k_\infty,\]
and
\begin{align*}
R_{n-1}+S_{n-1}+1-k+\mu_{m_{n-1}+M_{n-1}+1-k}&=R_{n-1}+S_{n-1}+1-k\\
&\quad+k_0+k_\infty-1-m_{n-1}-M_{n-1}\\
&=k'.
\end{align*}
\end{proof}

\begin{proposition}\label{R4}
Assume that $t \leq \min(r_n-1,s_n-1)$ and $M_{n-1}+1 \leq k \leq p^n-\max(-e_n,0)$.   Then $\lambda_\ell=(r_n+s_n-2t)p^n-\mu_{m_{n-1}+M_{n-1}+1-k}$, $(g-1)^{(r_n+s_n-2t-1)p^n}$ maps $D_{(r_n+s_n-t)p^n+R_{n-1}+S_{n-1}-k}$ to $D_{(t+1) p^n+R_{n-1}+S_{n-1}-k}$, and $D_{(t+1) p^n+R_{n-1}+S_{n-1}-k}$ is mapped to $D_{t p^n+k'}$ by $(g-1)^{p^n-\mu_{m_{n-1}+M_{n-1}+1-k}-1}=(g-1)^{p^n+m_{n-1}+M_{n-1}-k_0-k_\infty}$.    Assume that $C=J(R_{n-1},p^n-S_{n-1})$ and $B=J(p^n-R_{n-1},S_{n-1})$ are known.  
\begin{enumerate}
\item Define $z_\ell \in D_{(t+1) p^n+R_{n-1}+S_{n-1}-k}$ by
\begin{align*}
z_\ell&=\sum_{i=0}^t (-1)^i \sum_{j=1}^{k-S_{n-1}} c_{k-S_{n-1}+1-j,j}v_{i p^n+R_{n-1}+S_{n-1}+j-k,(t+1-i)p^n+1-j}\\
&\quad+(-1)^{R_{n-1}} \sum_{i=0}^t (-1)^i \sum_{j=1}^{k-R_{n-1}}b_{k-R_{n-1}+1-j,j}v_{(i+1)p^n+R_{n-1}+j-k,(t-i)p^n+S_{n-1}+1-j}
\end{align*}
Then $(g-1)^{p^n -\mu_{m_{n-1}+M_{n-1}+1-k}-1}(z_\ell)=x_{ t p^n+k'}$.  
\item Let $\zeta=M(4)^{-1} X$ and $\eta=M(3)^{-1}X$, define
$y_\ell \in D_{(r_n+s_n-t)p^n+R_{n-1}+S_{n-1}-k}$ by
\begin{align*}
y_\ell&=\sum_{i=0}^t \zeta_{i+1,1} \sum_{j=1}^{k-S_{n-1}} c_{k-S_{n-1}+1-j,j}v_{(r_n-t+i)p^n+R_{n-1}+S_{n-1}+j-k,(s_n-t)p^n+1-j}\\
&\quad+(-1)^{R_{n-1}} \sum_{i=0}^t \eta_{i+1,1} \sum_{j=1}^{k-R_{n-1}}b_{k-R_{n-1}+1-j,j}
v_{(r_n-t+i)p^n+R_{n-1}+j-k,(s_n-t)p^n+S_{n-1}+1-j}
\end{align*}
Then $(g-1)^{(r_n+s_n-2t-1)p^n}(y_\ell)=z_\ell$.
\end{enumerate}
\end{proposition}
\begin{proof}
(1) The coefficient $\alpha_a$ of $v_{a, t p^n+k'+1-a}$ in
$(g-1)^{p^n -\mu_{m_{n-1}+M_{n-1}+1-k}-1}(z_\ell)$ is
\begin{align*}
&\sum_{i=0}^t (-1)^i \sum_{j=1}^{k-S_{n-1}} c_{k-S_{n-1}+1-j,j}\binom{p^n -\mu_{m_{n-1}+M_{n-1}+1-k}-1}
{i p^n+R_{n-1}+S_{n-1}+j-k-a}\\
&\quad+(-1)^{R_{n-1}} \sum_{i=0}^t (-1)^i \sum_{j=1}^{k-R_{n-1}}b_{k-R_{n-1}+1-j,j}
\binom{p^n -\mu_{m_{n-1}+M_{n-1}+1-k}-1}{(i+1)p^n+R_{n-1}+j-k-a}.
\end{align*}
Recall that
\[V_{R_{n-1}} \otimes V_{p^n-S_{n-1}}
\cong \max(f_n,0) \cdot V_{p^n}\oplus
V_{p^n-\mu_{m_{n-1}}}\oplus \dots \oplus V_{p^n-\mu_{\max(e_n,0)+1}}.\]
When $R_{n-1}\leq S_{n-1}$, so $\max(f_n,0)=0$,  $V_{p^n-\mu_{m_{n-1}+1-(k-M_{n-1})}}$ is the
$k-M_{n-1}=k-S_{n-1}$ summand of $V_{R_{n-1}} \otimes V_{p^n-S_{n-1}}$.

When $R_{n-1}>S_{n-1}$, so $\max(f_n,0)=R_{n-1}-S_{n-1}$,$V_{p^n-\mu_{m_{n-1}+1-(k-M_{n-1})}}$ is the
$R_{n-1}-S_{n-1}+(k-M_{n-1})=k-S_{n-1}$ summand of $V_{R_{n-1}} \otimes V_{p^n-S_{n-1}}$.

Apply Theorem~\ref{ZThm} with $S_{n-1}$ replaced by $p^n-S_{n-1}$, $(k,k')$ replaced by $(k-S_{n-1},k'-S_{n-1})$, $k-e_n$ replaced by $k-S_{n-1}-(R_{n-1}-S_{n-1})=k-R_{n-1}$, $(A,D)$ replaced by $(C,B)$, $ t p^n+p^n-S_{n-1}+k'$ replaced by $t p^n+p^n-(p^n-S_{n-1})+k'-S_{n-1}=t p^n+k'$, and $\mu_k$ replaced by $p^n-\mu_{m_{n-1}+1-(k-M_{n-1})}$ to conclude that $\alpha_a=f_t(a)=(-1)^{a-1}$ when $1 \leq a \leq  t p^n+k'$.  We have proved our result.

(2) Let 
\[z'_\ell=\sum_{i=0}^t (-1)^i \sum_{j=1}^{k-S_{n-1}} c_{k-S_{n-1}+1-j,j}v_{i p^n+R_{n-1}+S_{n-1}+j-k,(t+1-i)p^n+1-j},\]
\[z''_\ell=\sum_{i=0}^t (-1)^i \sum_{j=1}^{k-R_{n-1}}b_{k-R_{n-1}+1-j,j}v_{(i+1)p^n+R_{n-1}+j-k,(t-i)p^n+S_{n-1}+1-j},\]
\[y'_\ell=\sum_{i=0}^t \zeta_{i+1,1} \sum_{j=1}^{k-S_{n-1}} c_{k-S_{n-1}+1-j,j}v_{(r_n-t+i)p^n+R_{n-1}+S_{n-1}+j-k,(s_n-t)p^n+1-j},\]
and 
\[y''_\ell=\sum_{i=0}^t \eta_{i+1,1} \sum_{j=1}^{k-R_{n-1}}b_{k-R_{n-1}+1-j,j}
v_{(r_n-t+i)p^n+R_{n-1}+j-k,(s_n-t)p^n+S_{n-1}+1-j},\]
so $z_\ell=z'_\ell+(-1)^{R_{n-1}} z''_\ell$ and $y_\ell=y'_\ell+(-1)^{R_{n-1}} y''_\ell$.  We will show that $(g-1)^{(r_n+s_n-2t-1)p^n}(y'_\ell)=z'_\ell$ and $(g-1)^{(r_n+s_n-2t-1)p^n}(y''_\ell)=z''_\ell$.

When $1 \leq i \leq t+1 $ the $(i,j)$ entry of $M(4)\zeta$ is
\[
(-1)^{i-1}=\sum_{j=1}^{t+1} \binom{r_n+s_n-2t-1}{r_n-t+j-i}\zeta_{j,1}=
\sum_{j=0}^t\binom{r_n+s_n-2t-1}{r_n-t+j+1-i} \zeta_{j+1,1}.
\]
By Corollary~\ref{p^n2}, since the only vectors in $\mathcal{B}_{(t+1)p^n+R_{n-1}+S_{n-1}-k}$ with nonzero coefficients in 
$(g-1)^{(r_n+s_n-2t-1)p^n}(v_{(r_n-t+i)p^n+R_{n-1}+S_{n-1}+j-k,(s_n-t)p^n+1-j})$ are \\$v_{a p^n+R_{n-1}+S_{n-1}+j-k,(t+1-a)p^n+1-j}$ with $0 \leq a \leq t$,
\begin{align*}
(g-1)&^{(r_n+s_n-2t-1)p^n}(v_{(r_n-t+i)p^n+R_{n-1}+S_{n-1}+j-k,(s_n-t)p^n+1-j})\\
&=\sum_{a=0}^t \binom{r_n+s_n-2t-1}{r_n-t+i-a}v_{a p^n+R_{n-1}+S_{n-1}+j-k,(t+1-a)p^n+1-j},
\end{align*}
using the fact that $r_n+s_n-2t-1<p$ by Lemma~\ref{M3}.
The two last computations combine in the usual way to show $(g-1)^{(r_n+s_n-2t-1)p^n}(y'_\ell)=z'_\ell$.

When $1 \leq i \leq t+1 $ the $(i,j)$ entry of $M(3) \eta$ is
\[(-1)^{i-1}=\sum_{j=1}^{t+1}\binom{r_n+s_n-2t-1}{r_n-t+j-i-1}\eta_{j,1}=\sum_{j=0}^{t}\binom{r_n+s_n-2t-1}{r_n-t+j-i}\eta_{j+1,1}.\]

By Corollary~\ref{p^n2}, since the only vectors in $\mathcal{B}_{(t+1)p^n+R_{n-1}+S_{n-1}-k}$ with nonzero coefficients in 
$(g-1)^{(r_n+s_n-2t-1)p^n}(v_{(r_n-t+i)p^n+R_{n-1}+j-k,(s_n-t)p^n+S_{n-1}+1-j})$ are 
\\$v_{(a+1)p^n+R_{n-1}+j-k,(t-a)p^n+S_{n-1}+1-j}$ with $0 \leq a \leq t$,
\begin{align*}
(g-1)&^{(r_n+s_n-2t-1)p^n}(v_{(r_n-t+i)p^n+R_{n-1}+j-k,(s_n-t)p^n+S_{n-1}+1-j})\\
&=\sum_{a=0}^t \binom{r_n+s_n-2t-1}{r_n-t+i-a-1}v_{(a+1)p^n+R_{n-1}+j-k,(t-a)p^n+S_{n-1}+1-j},
\end{align*}
using the fact that $r_n+s_n-2t-1<p$ by Lemma~\ref{M4}.
The two last computations combine in the usual way to show $(g-1)^{(r_n+s_n-2t-1)p^n}(y''_\ell)=z''_\ell$.
\end{proof}

\begin{lemma}\label{M5}
Suppose that $M(5)$ is $(t+1) \times (t+1)$ with $(i,j)$ entry  $\binom{r_n+s_n-2 t-2}{r_n-t+j-i-1}$ and that $m_{n-1}+M_{n-1}+1 \leq k \leq p^n$.  Then $M(5)$ is invertible and $r_n+s_n-2 t-2<p$.
\end{lemma}
\begin{proof}
By Lemma~\ref{pdiv} it suffices to show that $r_n-t-1 \leq r_n+s_n-2 t-2$ and $( r_n+s_n-2 t-2)+t=r_n+s_n-t-2<p$.

Since $t p^n+k \leq r_n p^n+R_{n-1}$ and $t p^n+k \leq s_n p^n+S_{n-1}$ and $k>\max(R_{n-1},S_{n-1})$, $t \leq \min(r_n-1,s_n-1)$.  Thus $r_n+s_n-2t-2=(r_n-t-1)+(s_n-t-1) \geq r_n-t-1$.

Since $t p^n+k>(r_n+s_n-p)p^n+m_{n-1}+M_{n-1}$ and $k >m_{n-1}+M_{n-1}$, $t \geq r_n+s_n-p$ because if $t \leq r_n+s_n-p-1$, then $t p^n+k \leq (r_n+s_n-p)p^n$.   It follows that $p \geq r_n+s_n-t >r_n+s_n-t-2$.  Hence $r_n+s_n-2 t-2<p$.
\end{proof}
Now we look at the case when $m_{n-1}+M_{n-1}+1 \leq k \leq p^n$, which only occurs when $m_{n-1}+M_{n-1}<p^n$. 
\begin{proposition}\label{R5}
Assume that $m_{n-1}+M_{n-1}+1 \leq k \leq p^n$ and $t \leq \min(r_n-1,s_n-1)$.  Then $\lambda_\ell=(r_n+s_n-2t-1)p^n$ by Corollary~\ref{RenaudCor}.  Hence $(g-1)^{(r_n+s_n-2t-2)p^n}$ maps $D_{(r_n+s_n-t)p^n+R_{n-1}+S_{n-1}-k}$ to
$D_{(t+2)p^n+R_{n-1}+S_{n-1}-k}$.
\begin{enumerate}
\item Define $z_\ell \in D_{(t+2)p^n+R_{n-1}+S_{n-1}-k}$ by
\[z_\ell=(-1)^{R_{n-1}}\sum_{i=0}^t (-1)^i (-1)^{k-S_{n-1}-1}v_{(i +1)p^n+R_{n-1}+S_{n-1}+1-k,(t-i+1)p^n}
\] 
Then $(g-1)^{p^n-1}(z_\ell)=x_{(t+1)p^n+R_{n-1}+S_{n-1}+1-k}$.
\item Let $\zeta=M(5)^{-1}X$.  Define
$y_\ell \in D_{(r_n+s_n-t)p^n+R_{n-1}+S_{n-1}-k}$ by
\[y_\ell=(-1)^{k+R_{n-1}-S_{n-1}-1}
\sum_{i=0}^t \zeta_{i+1,1}v_{(r_n-t+i)p^n+R_{n-1}+S_{n-1}+1-k,(s_n-i)p^n}.
\]
Then $(g-1)^{(r_n+s_n-2t-2)p^n}(y_\ell)=z_\ell$.
\end{enumerate}
\end{proposition}
Note that
\[z_\ell=(-1)^{R_{n-1}}\sum_{i=0}^t (-1)^i\sum_{j=1}^{k-m_{n-1}-M_{n-1}}d_{k-m_{n-1}-M_{n-1}+1-j,j}v_{(i +1)p^n+R_{n-1}+S_{n-1}+j-k,(t-i+1)p^n+1-j}\]
with a corresponding expression for $y_\ell$ since the $k'$th anti-diagonal of $D=J(p^n-R_{n-1},p^n-S_{n-1})$, with $1 \leq k' \leq p^n-m_{n-1}-M_{n-1}$,  is $((-1)^{R_{n-1}+k'-1},0,\dots,0)$, so $d_{k-m_{n-1}-M_{n-1},1}=(-1)^{k-S_{n-1}-1}$, but we do not need to make that recursive call here.
\begin{proof}
(1) 
 The coefficient $\alpha_a$ of $v_{a,(t+2)p^n+R_{n-1}+S_{n-1}+2-k-a}$ in $(g-1)^{p^n-1}(z_\ell)$ is
\[\sum_{i=0}^t (-1)^{i+k+R_{n-1}-S_{n-1}-1} \binom{p^n-1}{(i +1)p^n+R_{n-1}+S_{n-1}+1-k-a}.\]
Exactly one of these terms is nonzero.  So
\begin{align*}
\alpha_a&= (-1)^{i+k+R_{n-1}-S_{n-1}-1} \binom{p^n-1}{(i +1)p^n+R_{n-1}+S_{n-1}+1-k-a}\\
&=(-1)^{i+k+R_{n-1}-S_{n-1}-1} (-1)^{(i +1)p^n+R_{n-1}+S_{n-1}+1-k-a}\\
&=(-1)^{1-a}\\
&=(-1)^{a-1}.
\end{align*}
Here we are using the fact that if $p$ is odd, so is $i+(i+1)p^n$.

(2) When $1 \leq i \leq t+1$ the $(i,1)$ entry of $M(5) \zeta$ is
\[(-1)^{i-1}=\sum_{j=1}^{t+1} \binom{r_n+s_n-2t-2}{r_n-t+j-i-1}\zeta_{j,1}=\sum_{j=0}^t \binom{r_n+s_n-2t-2}{r_n-t+j-i}\zeta_{j+1,1}.\]
By Corollary~\ref{p^n2},  since the only vectors in $\mathcal{B}_{(t+2)p^n+R_{n-1}+S_{n-1}-k}$ with non-zero coefficients in 
$(g-1)^{(r_n+s_n-2t-2)p^n}(v_{(r_n-t+i)p^n+R_{n-1}+S_{n-1}+1-k,(s_n-i)p^n})$ are\\
$v_{(a+1)p^n+R_{n-1}+S_{n-1}+1-k,(t-a+1)p^n}$ with $0 \leq a \leq t$,
\begin{align*}
(g-1)&^{(r_n+s_n-2t-2)p^n}(v_{(r_n-t+i)p^n+R_{n-1}+S_{n-1}+1-k,(s_n-i)p^n})\\
&=\sum_{a=0}^t \binom{r_n+s_n-2t-2}{r_n-t+i-a-1}v_{(a+1)p^n+R_{n-1}+S_{n-1}+1-k,(t-a+1)p^n},
\end{align*}
using the fact that $r_n+s_n-2t-2<p$ by  Lemma~\ref{M5}.

Letting $z'_\ell=(-1)^{k+R_{n-1}-S_{n-1}-1}z_\ell$ and $y'_\ell=(-1)^{k+R_{n-1}-S_{n-1}-1}y_\ell$, we will show that
\[(g-1)^{(r_n+s_n-2t-2)p^n}(y'_\ell)=(-1)^{(-1)^{k+R_{n-1}-S_{n-1}-1}}z'_\ell.\]
Then
\begin{align*}
(g&-1)^{(r_n+s_n-2t-2)p^n}(y'_\ell)\\
&=\sum_{i=0}^t \zeta_{i+1,1}
\left(\sum_{a=0}^t \binom{r_n+s_n-2t-2}{r_n-t+i-a-1}v_{(a+1)p^n+R_{n-1}+S_{n-1}+1-k,(t-a+1)p^n}\right)\\
&=\sum_{a=0}^t \left(\sum_{i=0}^t \binom{r_n+s_n-2t-2}{r_n-t+i-a-1}\zeta_{i+1,1}\right)
v_{(a+1)p^n+R_{n-1}+S_{n-1}+1-k,(t-a+1)p^n}\\
&=\sum_{a=0}^t (-1)^av_{(a+1)p^n+R_{n-1}+S_{n-1}+1-k,(t-a+1)p^n}\\
&=z'_\ell.
\end{align*}
\end{proof}

A careful examination of the results in this section shows that our choice of $y_\ell$ (and $z_\ell$)  when $|\{i \mid \lambda_i=\lambda_\ell\}|>1$  is consistent with Lemma~\ref{Mult} in Section~\ref{J}.  For example in Proposition~\ref{R1}, the presence of the term $(-1)^{R_{n-1}-k}$ ensures this.  And in Proposition~\ref{R2}, the fact that $a_{k+1-j,j}$ from $J(R_{n-1},S_{n-1})$ and $d_{k-e_n+1-j,j}$ from $J(p^n-R_{n-1},p^n-S_{n-1})$ meet the requirement for $\mu_k$ means $y_\ell$ does too.

\section{$z_\ell$ and $y_\ell$ when $m_{n-1}=0$}\label{m=0}

We are still assuming that $\max(r+s-p^{n+1},0)+1 \leq \ell=t p^n+k \leq \min(r,s)$ where $t$ and $k$ are nonnegative integers with $1 \leq k \leq p^n$.   The matrices $M(3)$, $M(4)$, and $M(5)$ were all defined in Section~\ref{m>0} and $X$ will the $(t+1) \times 1$ matrix over $F$ with $(i,1)$ entry $(-1)^{i-1}$.

\begin{lemma}\label{M=0}
Assume that $M_{n-1}=0$.    In this case $r=r_np^n$ and $s=s_np^n$.   Assume that $t$ and $k$ are integers satisfying 
$0 \leq t  \leq \min(r_n-1,s_n-1)$ and $1 \leq k \leq p^n$.  Then $\lambda_\ell=(r_n+s_n-2t-1)p^n$ by Theorem~\ref{RenaudThm}.
\begin{enumerate}
\item Define $z_\ell \in D_{(t+2)p^n-k}$ by
\[
z_\ell=\sum_{i=0}^{t}(-1)^i(-1)^{k-1}v_{(i+1)p^n+1-k,(t-i+1)p^n}.
\]
Then $(g-1)^{p^n-1}(z_\ell)=x_{(t+1)p^n+1-k}$.
\item Let $\zeta=M(5)^{-1}X$.  Define $y_\ell \in D_{(r_n+s_n-t)p^n-k}$ by
\[y_\ell=\sum_{i=0}^t \zeta_{i+1,1}(-1)^{k-1}v_{(r_n-t+i)p^n+1-k,(s_n-i)p^n}.
\]
Then $(g-1)^{ (r_n+s_n-2 t-2)p^n}(y_\ell)=z_\ell$.
\end{enumerate}
\end{lemma}

This result is a special case of Proposition~\ref{R5}.

\begin{lemma}\label{M=R1}
Assume that $m_{n-1}=0$,  and $M_{n-1}=R_{n-1}>0$.  In this case $r=r_n p^n+R_{n-1}$ and $s=s_n p^n$.  Assume that $t$ and $k$ are integers satisfying $0 \leq t \leq \min(r_n,s_n-1)$ and $1 \leq k \leq R_{n-1}$.   Then $\lambda_\ell=(r_n+s_n-2t)p^n$ by Theorem~\ref{RenaudThm}.
\begin{enumerate}
\item Define $z_\ell \in D_{(t+1)p^n+R_{n-1}-k}$  by
\[
z_\ell=\sum_{i=0}^t (-1)^i (-1)^{R_{n-1}-k}v_{i p^n+R_{n-1}+1-k,(t-i+1)p^n}.
\]
Then  $(g-1)^{p^n-1}(z_\ell)=x_{t p^n+R_{n-1}+1-k}$.
\item Let $\zeta=M(4)^{-1}X$.  Define $y_\ell \in D_{(r_n+s_n-t)p^n+R_{n-1}-k}$
by
\[y_\ell=\sum_{i=0}^t  \zeta_{i+1,1} (-1)^{R_{n-1}-k}v_{(r_n-t+i)p^n+R_{n-1}+1-k,(s_n-i)p^n}
\]
Then $(g-1)^{(r_n+s_n-2t-1)p^n}(y_\ell)=z_\ell$.
\end{enumerate}
\end{lemma}

This result is a special case of Proposition~\ref{R3AndR>S}.

\begin{lemma}\label{M=R2}
Assume that $m_{n-1}=0$,  and $M_{n-1}=R_{n-1}>0$.  In this case $r=r_n p^n+R_{n-1}$ and $s=s_n p^n$.  Assume that $t$ and $k$ are integers satisfying $0 \leq t \leq \min(r_n,s_n-1)$ and $R_{n-1}+1 \leq k \leq p^n$.   Then $\lambda_\ell=(r_n+s_n-2t-1)p^n$ by Theorem~\ref{RenaudThm}.
\begin{enumerate}
\item Define $z_\ell \in D_{(t+2)p^n+R_{n-1}-k}$ by
\[z_\ell=(-1)^{R_{n-1}}\sum_{i=0}^t (-1)^{i +k-1} v_{(i+1)p^n+R_{n-1}+1-k,(t-i+1)p^n}.
\]
Then $(g-1)^{p^n-1}(z_\ell)=x_{(t+1)p^n+R_{n-1}+1-k}$.
\item Let $\zeta=M(5)^{-1}X$.  Define $y_\ell \in D_{(r_n+s_n-t)p^n+R_{n-1}-k}$
by
\[y_\ell=\sum_{i=0}^t  \zeta_{i+1,1}(-1)^{R_{n-1}+k-1}v_{(r_n-t+i)p^n+R_{n-1}+1-k,(s_n-i)p^n}.
\]
Then $(g-1)^{(r_n+s_n-2t-2)p^n}(y_\ell)=z_\ell$.
\end{enumerate}
\end{lemma}

This result is a special case of Proposition~\ref{R5}.

\begin{lemma}\label{M=S1}
Assume that $m_{n-1}=0$ and $M_{n-1}=S_{n-1}>0$.  In this case $r=r_n p^n$ and $s=s_n p^n+S_{n-1}$.  Assume that  $t$ and $k$ are integers satisfying $0 \leq t \leq \min(r_n-1,s_n)$ and $1 \leq k \leq S_{n-1}$.   Then $\lambda_\ell=(r_n+s_n-2t)p^n$ by Theorem~\ref{RenaudThm}.
\begin{enumerate}
\item Define $z_\ell \in D_{(t+1)p^n+S_{n-1}-k}$ by
\[
z_\ell=\sum_{i=0}^t (-1)^{i+k-1}v_{(i+1)p^n+1-k,(t-i)p^n+S_{n-1}}.
\]
Then $(g-1)^{p^n-1}(z_\ell)=x_{t p^n+S_{n-1}+1-k}$.
\item Let $\zeta=M(3)^{-1}X$.  Define $y_\ell \in D_{(r_n+s_n-t)p^n+S_{n-1}-k}$
by
\[y_\ell=(-1)^{k-1}\sum_{i=0}^t \zeta_{i+1,1}v_{(r_n-t+i)p^n+1-k,(s_n-i)p^n+S_{n-1}}.
\]
Then $(g-1)^{(r_n+s_n-2t-1)p^n}(y_\ell)=z_\ell$.
\end{enumerate}
\end{lemma}

This result is a special case of Proposition~\ref{R3AndR<S}.

\begin{lemma}\label{M=S2}
Assume that $m_{n-1}=0$ and $M_{n-1}=S_{n-1}>0$.  In this case $r=r_n p^n$ and $s=s_n p^n+S_{n-1}$.  Assume that  $t$ and $k$ are integers satisfying $0 \leq t \leq \min(r_n-1,s_n)$ and $S_{n-1}+1 \leq k \leq p^n$.   Then $\lambda_\ell=(r_n+s_n-2t-1)p^n$ by Theorem~\ref{RenaudThm}.
\begin{enumerate}
\item Define $z_\ell \in D_{(t+2)p^n+S_{n-1}-k}$ by
\[
z_\ell=\sum_{i=0}^t (-1)^{i+k-S_{n-1}-1}v_{(i+1)p^n+S_{n-1}+1-k,(t-i+1)p^n}.
\]
Then $(g-1)^{p^n-1}(z_\ell)=x_{(t+1)p^n+S_{n-1}+1-k}$.
\item Let $\zeta=M(5)^{-1}X$.   Define $y_\ell \in D_{(r_n+s_n-t)p^n+S_{n-1}-k}$
by
\[y_\ell=(-1)^{k-S_{n-1}-1}\sum_{i=0}^t \zeta_{i+1,1} v_{(r_n-t+i)p^n+S_{n-1}+1-k,(s_n-i)p^n}.\]
Then $(g-1)^{(r_n+s_n-2t-2)p^n}(y_\ell)=z_\ell$.
\end{enumerate}
\end{lemma}

This is a special case of Proposition~\ref{R5}.

\end{document}